\documentclass[oneside, a4paper,11pt,reqno]{amsart}
\usepackage{xcolor}
\usepackage{ulem}
\usepackage{color}

\usepackage[utf8]{inputenc}

\newcommand{\p} {\textnormal{\textsf{P}}}

\newtheorem{hypo}{Hypothesis}

\newtheorem{prop}[hypo]{Proposition}

\newtheorem{thm}[hypo]{Theorem}

\newtheorem{lem}[hypo]{Lemma}

\newtheorem{rqe}[hypo]{Remark}

\newtheorem{coro}[hypo]{Corollary}

\def\PP{\mathbb{P}}

\def\EE{\mathbb{E}}

\let\BFseries\bfseries\def\bfseries{\BFseries\mathversion{bold}} 

\def\deq{\stackrel{\stackrel{\mathcal{L}}{=}}{}}
\def\onotation{O}

\def \ind {{\bf 1}}

\newcommand {\acc}[1] {\left\{ {#1} \right\}}

\def\dd{\mbox{d}}
\def\P{\mathbb{P}}

\title{Persistence probabilities for stationary increment processes}

\begin{document}
 \author{Frank Aurzada} 
\address{AG Stochastik, Fachbereich Mathematik, Technische Universit\"at Darmstadt, Schlossgartenstr.\ 7, 64289 Darmstadt, Germany.}
\email{aurzada@mathematik.tu-darmstadt.de}

\author{Nadine Guillotin-Plantard} 
\address{Institut Camille Jordan, CNRS UMR 5208, Universit\'e de Lyon, Universit\'e Lyon 1, 43, Boulevard du 11 novembre 1918, 69622 Villeurbanne, France.}
\email{nadine.guillotin@univ-lyon1.fr}

\author{Fran\c{c}oise P\`ene}
\address{Universit\'e de Brest and IUF,
LMBA, UMR CNRS 6205, 29238 Brest cedex, France}
\email{francoise.pene@univ-brest.fr}


\begin{abstract}
We study the persistence probability for processes with stationary increments. Our results apply to a number of examples: sums of stationary correlated random variables whose scaling limit is fractional Brownian motion; random walks in random sceneries; random processes in Brownian scenery; and the Matheron-de Marsily model in $\mathbb Z^2$ with random orientations of the horizontal layers.
Using a new approach, strongly related to the study of the range, we obtain an upper bound of the optimal order in general and improved lower bounds (compared to previous literature) for many specific processes. 
 \end{abstract}

\maketitle

\section{Introduction}
Persistence concerns the probability that a stochastic process has a long negative excursion. In this paper, we are concerned mainly with discrete-time processes. If $Z=(Z_n)_{n\ge 0}$ is a stochastic process defined on a probability space $(\Omega,\mathcal F,\mathbb P)$, we study the rate of decay of the probability
$$
\PP\left( Z_n^* \leq a \right),\qquad \text{as $n\to+\infty$},\quad\mbox{with}\quad Z_n^*:= \max_{k=1,\ldots, n}Z_k\, ,
$$
for $a\in\mathbb R$.
In many cases of interest, the above probability decreases polynomially, i.e.\ as $n^{-\theta+o(1)}$, and it is the first goal to find the persistence exponent $\theta$. For a recent overview on this subject, we refer to the survey \cite{AS} and for the relevance in theoretical physics we recommend \cite{Maj1,BMS13}. 

For random walks, L\'evy processes, and a number of other Markov processes, persistence probabilities are well studied. E.g.\ for random walks, one can obtain the persistence probabilities with the help of the more general fluctuation theory.
Since \cite{Molchan1999}, persistence probabilities for intrinsically non-Markovian processes such as fractional Brownian motion are investigated. Methods based on exponential functionals have been
developed to study persistence probabilities in such contexts (see e.g.\ \cite{BFFN,castellguillotinwatbled,Aurzada}). In this paper, we present a new and simple approach adapted to the case  of processes with stationary increments. Our method, based on the study of the expectation of the number of elements in the range, provides results in
non-Markovian contexts, and does not require the existence of exponential moments.

The purpose of this paper is to analyse the persistence probability for stationary increment discrete time processes, i.e.\ processes such that for any nonnegative integer $k$, $(Z_{n+k}-Z_k)_{n\ge 0} \deq (Z_n)_{n\ge 0}$, where $\deq$ means equality in law (note that
this implies that $Z_0=0$).
Under rather general assumptions, we prove that 
\[
\PP\left( Z_n^* \leq a \right)\approx \mathbb E\left[Z_n^*\right]/n\,\qquad \text{as }n\to+\infty,
\]
where $\approx$ means up to a multiplicative term in $n^{o(1)}$ (the multiplicative term is bounded by a constant for the upper bound and, for the lower bound, is larger than a function that is slowly varying at infinity).
We emphasize the fact that we obtain
the exact order when the increments are bounded
and that we obtain estimates even if the increments admit no exponential moment. 
Further, we stress that $\EE [Z_n^*]$ is in many cases a well-accessible quantity. All our conditions are formulated in terms of simple quantities of the process.
Stationary increments are a feature shared by many stochastic processes that are important in theory and applications, and we shall treat here a number of examples. 

The outline of the paper is as follows. Our main general results are stated and proved in Section~\ref{sec:gene}. 
These general results are then applied to the following examples: Sums of stationary sequences and fractional Brownian motion are treated in Section~\ref{sec:resultssumsof}, random walks in random scenery in Section~\ref{sec:rwrs}, random processes in Brownian scenery in Section~\ref{RPBS}, and the Matheron-de Marsily model in Section~\ref{sec:mdm}.

\section{General results for stationary increment processes}\label{sec:gene}

When dealing with persistence probabilities
$\mathbb P\left(Z_n^*\le a\right)$, a natural question 
is the choice of the level $a$. In most of the cases, going from a
positive level to another positive level can be simply done by multiplying $Z$ by some
positive constant.
The next technical lemma allows to pass from a negative level to a positive one and conversely, up to a change
of the multiplicative constants in the bounds. This will be crucial as we often derive upper bounds for negative levels and lower bounds for positive levels, which have to be brought to matching.
\begin{lem}\label{lem:positiveassociation}
Assume that
there exists a sub-$\sigma$-algebra $\mathcal F_0$ of $\mathcal F$
such that, given $\mathcal F_0$, the increments of $Z$ are positively associated 
and that their common conditional distribution is independent of $\mathcal F_0$. Then, for any 
$a\ge 0$ and $m>0$
such that $\mathbb P(Z_1\le -m)>0$,
$$\mathbb P(Z_n^*\le -m)\le \mathbb P\left(Z_n^*\le a\right)\le
        \frac {\mathbb P(Z^*_{n+\lceil a/m\rceil+1}\le -m)} {(\mathbb P(Z_1\le -m))^{\lceil a/m\rceil +1}} .$$
\end{lem}
Note that this lemma does not allow to go 
from a result of the strong asymptotic form $\mathbb P(Z_n^*\le -1)\sim p_n$ to an analogous result for 
$\mathbb P(Z_n^*\le 1)$.
\subsection{Upper bounds for the persistence probability}
We start with our first central result, which is an upper bound for negative levels for centered processes, i.e. processes
with zero mean.
\begin{thm}\label{THMlem}
Let $(Z_n)_{n\ge 0}$ be a centered process with stationary increments.
Then
$$
\forall a>0~:\quad 
\mathbb P\left(Z_n^*\le -a\right)\le 
\frac {\mathbb E\left[Z_{n+1}^*\right]}{a n}.
$$
\end{thm}
When $Z_1$ is $\mathbb Z$-valued, 
a natural upper bound
is provided by the 
number of elements
$\mathcal R_n$ in the range of $Z$
up to time $n$, which is given by 
$\mathcal R_n:=\#\{Z_1,\ldots,Z_n\}$ (by convention, $\mathcal R_0=0)$.
In this context, a related quantity is $\mathbb P(T_0>n)$ with
$T_0:=\inf\{n\ge 1\ :\ Z_n = 0\}$ the first return time of $Z$ to $0$. 
Indeed
$$\mathbb P(Z_n^*\le -1)\le  \mathbb P(T_0>n)$$
and even $\mathbb P(Z_n^*\le -1) =  \mathbb P(T_0>n)/2$
if the distribution of $(Z_n)_n$ is symmetric and if $\mathbb P(Z_1\in\{-1,0,1\})=1$.
Under some general assumptions, we obtain 
a precise estimate for the tail distribution of $T_0$ of the form $\mathbb P(T_0>n)\sim \gamma\mathbb E[ \mathcal R_n]/n$. 
Moreover, in the very particular case of a process $(Z_n)_n$
moving to the nearest neighbours on $\mathbb Z$, $\mathcal R_n=Z_n^*+(-Z)_n^*+1$.
\begin{thm}[Integer valued processes] \label{THMZ}
Let $(Z_n)_{n\ge 0}$ be a $\mathbb Z$-valued 
process with stationary increments.
Then
\begin{equation}\label{majoA}
\mathbb P(T_0>n)\le \frac{\mathbb E[\mathcal R_n]}n.
\end{equation}
Further, 
\begin{itemize}
\item if there exists $y>1$ such that $\liminf_{n\rightarrow +\infty}\frac{\mathbb E[\mathcal R_{\lfloor yn\rfloor}]}{\mathbb E[\mathcal R_{n}]}>1$, then $$\liminf_{n\rightarrow +\infty} \frac n{\mathbb E[\mathcal R_n]}\mathbb P(T_0>n)>0.$$
\item if the sequence $(\mathbb E[\mathcal R_n])_{n\ge 1}$
is regularly varying with exponent $\gamma\in(0,1)$,
then
\begin{equation}\label{gammaa2}
\lim_{n\rightarrow +\infty}\frac n{\mathbb E[\mathcal R_n]}\mathbb P(T_0>n) =\gamma\, ;
\end{equation}
\item if $\mathbb P(Z_1\in\{-1,0,1\})=1$ and if the processes $(Z_n)_n$ and $(-Z_n)_n$ have same distribution and if \eqref{gammaa2} holds with $\gamma\in(0,1)$, then
\begin{equation}\label{EEE3ba}
\mathbb P\left(Z^*_n\le -1\right)=\frac 12\mathbb P(T_0>n)
   \sim \frac {\gamma\, \mathbb E[\mathcal R_n]} {2n} \, .
\end{equation}
\end{itemize}
\end{thm}


%
%

\subsection{Lower bounds for the persistence probability}
The lower bound is much more complicated to 
establish in the general case and will require additional assumptions.
\begin{thm}\label{thm:lower}
Let $(Z_n)_{n\ge 0}$ be a centered process with stationary increments. Assume that there is an $\varepsilon>0$ such that 
\begin{equation} \label{eqn:notregvar}
\rho:=\liminf_{n\rightarrow +\infty}\frac{\mathbb E[Z^*_{n+\lfloor\varepsilon n\rfloor}]}{\mathbb E[Z_n^*]}
>1.
\end{equation}
\begin{itemize}
\item[(i)] If $\mathbb P(-Z_1\le b)=1$, then
$$\liminf_{n\rightarrow +\infty}\frac{\varepsilon n}{\mathbb E[Z_n^*]}\mathbb P(Z_n^*<0)\ge \frac{\rho-1}b .$$
\end{itemize}
Assume from now on that $(b_n)_n$ is a sequence of positive numbers such that
\begin{equation}\label{eqn:hypbn}
K:=\limsup_{n\rightarrow +\infty} n\mathbb P(-Z_1>b_n)<+\infty\, ,
\end{equation}
and that at least one of the two following assumptions holds true:
\begin{itemize}
\item[(ii)] There exists $p>1$ such that
\begin{equation}\label{Hypkappa}
\kappa:=\limsup_{n\rightarrow +\infty} (\mathbb E[Z_n^*])^{-1}\Vert Z^*_{n+\lfloor \varepsilon n\rfloor}-Z_{n+\lfloor\varepsilon n\rfloor}\Vert_p <+\infty\, .
\end{equation}
\item[(iii)] The assumptions of Lemma~\ref{lem:positiveassociation} hold and
\begin{equation}\label{AAA}
\rho>\limsup_{n\rightarrow +\infty}\frac{\mathbb E\left[\max\left(0,Z^*_{\lfloor\varepsilon n\rfloor+1,n+\lfloor \varepsilon n\rfloor}\right)\right]}{\EE[ Z_n^*]}.
\end{equation}
This last condition holds true if 
$\rho>1+\limsup_{n\rightarrow +\infty}\frac{\mathbb E\left[\max(0,Z_{\lfloor \varepsilon n\rfloor+1})\right]}{\mathbb E[Z_n^*]}$.
\end{itemize}
Then there is an integer $d$ such that
$$
\liminf_{n\rightarrow +\infty}
      \frac {nb_{dn}}{\mathbb E[Z_n^*]}  \mathbb P\Big(Z_n^*< 0\Big)>0\, .$$
\end{thm}
The additional conditions in Theorem~\ref{thm:lower} are merely technical regularity conditions. 
We stress that they are easily checkable in terms of the following three accessible quantities: 
the asymptotic behaviour of $\EE[ |Z_n^*-Z_n|^p]$, of $\EE[Z_n^*]$ and the tail of $Z_1$. In particular, (\ref{eqn:notregvar}) is satisfied (for any $\varepsilon>0$) if the sequence $(\EE[Z_n^*])_n$ is regularly varying with positive exponent and in that case (\ref{Hypkappa}) is true as soon as $\limsup_{n\to+\infty}(\mathbb E[Z_n^*])^{-1}\Vert Z^*_{n}-Z_n\Vert_p <+\infty$ (since $\mathbb E[Z_{n+\lfloor \varepsilon n\rfloor}^*]/\mathbb E[Z_n^{*}]\sim (1+\varepsilon)^\alpha)$, for some $\alpha>0$)
and so as soon as $\limsup_{n\rightarrow +\infty}(\mathbb E[Z_n^*])^{-1}\Vert \max_{k=1,...,n}|Z_k|\Vert_p<\infty$.

\newpage
Let us first prove a sufficient criterion for condition \eqref{AAA}.
\begin{lem}\label{lem:ConditionAAA}
Assume that $((n^{-\alpha}Z_{\lfloor nt\rfloor})_t)_n$ converges in distribution to some 
$\alpha$-self similar process $(\Delta_t)_t$ with stationary increments and such that
$\mathbb E\left[Z^*_{n}\right]\sim n^\alpha\mathbb E\left[\sup_{[0,1]}\Delta\right]$ and
$\mathbb E\left[\max\left(0,Z^*_{\lfloor a n\rfloor+1,\lfloor b n\rfloor}\right)\right]\sim n^\alpha\EE\left[\max(0,\sup_{[a,b]}\Delta)\right]$ for every $a<b$ and that
$\mathbb P \left(\sup_{[0,1]} \Delta=0\right)<\mathbb P\left(\Delta_1<0\right)$.
Then there exists an $\varepsilon>0$ such that \eqref{AAA} holds true.\\
\end{lem}
Regarding the last assumption, we remark that one may indeed have $\mathbb P \left(\sup_{[0,1]} \Delta=0\right)>0$ for a self-similar process with stationary increments. Take, for example, a stable subordinator $(A_t)_{t\geq 0}$ (e.g.\ $A_t:=\inf \{ s>0 | B_s>t \}$ for a Brownian motion $(B_t)_t$) and an independent random variable $\eta\in\{\pm 1\}$. Then $\Delta_t := \eta A_t$ is self-similar, has stationary increments, and satisfies  $\mathbb P(\sup_{[0,1]} \Delta=0) = \mathbb P(\eta=-1)=\mathbb P\left(\Delta_1<0\right)$.

We discuss now assumption \eqref{eqn:hypbn}.
\begin{rqe}\label{MAINTHM}
If $\mathbb E[G(\max(-Z_1,0))]<\infty$ for some positive strictly increasing function $G$, then the Markov inequality gives
$$
\forall b>0~:\quad\mathbb P(-Z_1>b)=\mathbb P(\max(-Z_1,0)>b)\le \frac{\mathbb E[G(\max(-Z_1,0))]}{G(b)}.$$
Hence, if $\mathbb E[G(\max(-Z_1,0))]<\infty$  with e.g.\ $G(t)=e^{at}$, $G(t)=t^{1+a}$, or $G(t)=e^{at^2}$ for some $a>0$ then \eqref{eqn:hypbn} holds with respectively $b_n = c_0\log (n)$, 
$b_n=c_0 n^{\frac 1{1+a}}$, or $b_n = c_0\sqrt{\log (n)}$, for some suitable $c_0$.
\end{rqe}
In particular, if $Z_1$ admits moments of every order, then \eqref{eqn:hypbn} holds with $b_n=n^{o(1)}$.
Indeed, let us consider $\tilde b_n$ as the smallest $b\ge 0$ such that
$\mathbb P(-Z_1>b)\le 1/n$, then, for every $\gamma>1$,
$$
\mathbb P \left(-Z_1\ge  n^{\frac 1\gamma}\Vert \max(-Z_1,0)\Vert_\gamma\right)= \mathbb P \left((\max(-Z_1,0))^\gamma\ge  n\Vert \max(-Z_1,0)\Vert^\gamma_\gamma\right)
\le \frac 1n\, ,
$$
due to the Markov inequality. This shows that $\tilde b_n\le n^{\frac 1\gamma}\Vert \max(-Z_1,0)\Vert_\gamma$, for every 
 $\gamma>1$. This implies that $\limsup_n\log \tilde b_n/\log n\le 0$, i.e. $\tilde b_n\le n^{o(1)}$.
\begin{coro}\label{thmbounded}
Let $(Z_n)_{n\ge 0}$ be a centered process with stationary increments such that $\mathbb P(Z_1\ge -b)=1$ for some $b>0$. 
Assume moreover that
$(\mathbb E[Z_n^*])_n$ is $\gamma$-regularly
varying with $\gamma\in(0,1)$. Then
\[
\liminf_{n\rightarrow +\infty}\frac{n\mathbb P(Z_n^*<0)}{\mathbb E[Z_n^*]}\ge\frac{\gamma}{b}\, .
\]
\end{coro}

\subsection{Proofs of the general results} \label{sec:proofs}
\begin{proof}[Proof of Lemma~\ref{lem:positiveassociation}]
Let $K\ge 0$ be such that $Km\ge a$, so that $-(K+1)m+a\le -m$. Let $n$ be a positive integer.
Given $\mathcal F_0$, the increments of $(Z_n)_ {n\ge 0}$ are positively associated, hence
\begin{eqnarray*}
&& \mathbb P\left(\left.Z_{n+K+1}^*\le-m\right|\mathcal F_0\right)\\
&\ge&
\mathbb P\left(\left.\forall k=1,\ldots,K+1,\ (Z_k-Z_{k-1})\le -m,\ \max_{\ell=K+2,\ldots,n+K+1} (Z_\ell-Z_{K+1})\le a\right|\mathcal F_0\right)\\
&\ge& \left(\prod_{k=1}^{K+1}\mathbb P\left(\left.Z_{k}-Z_{k-1}\le -m\right|\mathcal F_0\right)\right)\,
      \mathbb P\left(\left.\max_{\ell=K+2,\ldots,n+K+1} (Z_\ell-Z_{K+1})\le a\right|\mathcal F_0\right)\\
&\ge&  \left(\mathbb P(Z_1\le -m)\right)^{K+1}\,      \mathbb P\left(\left.\max_{\ell=K+2,\ldots,n+K+1} (Z_\ell-Z_{K+1})\le a\right|\mathcal F_0\right),
\end{eqnarray*}
since the conditional distribution of $Z_k-Z_{k-1}$ given $\mathcal F_0$ is the
distribution of $Z_1$. So
\[
\mathbb P\left(Z^*_{n+K+1} \le-m\right)
\ge\left(\mathbb P(Z_1\le -m)\right)^{K+1}\,
\mathbb P\left( Z_n^*\le a\right),
\]
since the increments of $Z$ are stationary.
\end{proof}
\begin{proof}[Proof of Theorem~\ref{THMlem}]
Up to dividing $Z$ by $a$, we assume that $a=1$. We set 
$Z_{n,m}^*:=\max_{k=n,\ldots,m}Z_k$. Since the increments of $Z$ are stationary, we deduce that $Z_{k+1,k+n}^*-Z_k$
has the same distribution as $Z_{n}^*$ and so
\begin{eqnarray*}
n\,\mathbb P(Z_{n}^*\le -1)&\le&
     \sum_{k=1}^{n}\mathbb P(Z_{n-k+1}^*\le -1)\\
  &=& \sum_{k=1}^{n}\mathbb P(1+Z_{k+1,n+1}^*\le Z_{k})\\
 &\le&\mathbb E[\#\mathcal E_n]\, ,
\end{eqnarray*}
with $\mathcal E_n:=\{k=1,\ldots,n\, :\, 1+Z_{k+1,n+1}^*\le Z_k\}$.
Note that
$$\# \mathcal E_n\le
 \sum_{k\in\mathcal E_n}(Z_k-Z_{k+1,n+1}^*)
    \le \sum_{k=1}^n(Z_{k,n+1}^*-Z_{k+1,n+1}^*)\le Z_{n+1}^*-Z_{n+1}.$$Hence
$$n\,\mathbb P(Z_{n}^*\le -1) \le \mathbb E[Z_{n+1}^*],$$
since
$\mathbb E[Z_{n+1}]=0.$
\end{proof}
Similar to the fact stated in Theorem~\ref{THMlem}, one can obtain the following comparable formula.
\begin{prop} Let $(Z_n)_{n\ge 0}$ be a centered process with stationary increments. Then 
$$
\forall a>0~:\quad
\mathbb P\left(\min_{\ell=1,\ldots,n}|Z_\ell|\ge a\right)\le 
\frac{\mathbb E\left [Z_{n+1}^*+(-Z)_{n+1}^*\right]+a}{an}.
$$
\end{prop}
\begin{proof} By multiplying $Z$ by a constant, one can assume w.l.o.g.\ that $a=1$. Since $Z$ has stationary increments, 
\begin{eqnarray*}
n \mathbb P\left(\min_{\ell=1,\ldots,n} | Z_\ell|\ge 1\right)&\le&
\sum_{k=1}^n\mathbb P\left(\min_{\ell=1,\ldots,k} | Z_\ell|\ge 1\right)\\
&=&\sum_{k=1}^n\mathbb P\left(\min_{\ell=1,\ldots,k}  |Z_{n-k+1+\ell}-Z_{n-k+1}|\ge 1\right)\\
&\le&\sum_{k=1}^n\mathbb P\left(\forall \ell=1,\ldots,k ~: \lfloor Z_{n-k+1+\ell}\rfloor\ne \lfloor Z_{n-k+1}\rfloor\right)\, .
\end{eqnarray*}
Now, observing that the event
$\{\forall \ell=1,...,k\, :\, \lfloor Z_{n-k+1+\ell}\rfloor\ne \lfloor Z_{n-k+1}\rfloor\} $ means that $n-k+1$ is the last visit time of $\lfloor Z\rfloor$
to $\lfloor Z_{n-k+1}\rfloor$ up to time $n+1$, we conclude that
\begin{eqnarray*}
n \mathbb P\left(\min_{\ell=1,\ldots,n} | Z_\ell|\ge 1\right)
&\leq & 
\mathbb E\left[ \#\{ \lfloor Z_1\rfloor,\dots,  \lfloor Z_{n+1}\rfloor\}-1 \right]\\
&\le&\mathbb E\left[\max_{k=1,\ldots,n+1}\lfloor Z_{k}\rfloor-\min_{k=1,\ldots,n+1}\lfloor 
Z_{k}\rfloor\right]\\
&\le&\mathbb E\left[\max_{k=1,\ldots,n+1} Z_{k}-\min_{k=1,\ldots,n+1} 
Z_{k}+1\right].
\end{eqnarray*}
\end{proof}

\begin{proof}[Proof of Theorem \ref{THMZ}]
Let us prove \eqref{majoA}. Note that since $(Z_\ell)_\ell$ is $\mathbb Z$-valued and has stationary increments, 
considering the last visit times of $Z$ up to time $n$, we have
\begin{eqnarray}
 \mathbb E [\mathcal R_n] &=& \mathbb E\left[ \#\{ Z_1, \ldots, Z_n\}\right]
\notag \\
&=& \sum_{k=0}^{n-1} \mathbb P( \forall \ell =1,\ldots, k : Z_{n-k+\ell}\neq Z_{n-k})
\notag \\
&=& \sum_{k=0}^{n-1} \mathbb P( \forall \ell =1,\ldots, k : Z_{\ell}\neq 0)
\notag \\
&=& \sum_{k=0}^{n-1} \mathbb P( T_0 > k)\, . \label{eqn:alforg}
\end{eqnarray}
Since $(\mathbb P(T_0>k))_k$ is non increasing,  for every $0\leq x<1<y$, we have by (\ref{eqn:alforg}), 
\begin{equation}\label{Eqxy}
\frac{\mathbb E[\mathcal R_{\lfloor y n\rfloor}-\mathcal R_{n}]}{\lfloor yn\rfloor- n}  \le  \mathbb P(T_0>n)
\le \frac{\mathbb E[\mathcal R_{ n}-\mathcal R_{\lfloor x n\rfloor}]}{n-\lfloor xn\rfloor}\, ;
\end{equation}
from which we get \eqref{majoA} (by taking $x=0$).
If $\liminf_{n\rightarrow +\infty}\frac{\mathbb E[\mathcal R_{\lfloor yn\rfloor}]}{\mathbb E[\mathcal R_{n}]}>1$,
then
$$\liminf_{n\rightarrow +\infty}\frac n{\mathbb E[\mathcal R_n]}\mathbb P(T_0>n)\ge \left(\liminf_{n\rightarrow +\infty}\frac{\mathbb E[\mathcal R_{\lfloor yn\rfloor}]}
{\mathbb E[\mathcal R_{n}]}-1\right)\frac 1{y-1}>0   .$$
Assume now that $(\mathbb E[\mathcal R_n])_n$ is $\gamma$-regularly varying then \eqref{Eqxy} leads to
$$ \frac {y^\gamma -1}{y-1}\le\liminf_{n\rightarrow +\infty}\frac n{\mathbb E[\mathcal R_n]}\mathbb P(T_0>n)\le
    \limsup_{n\rightarrow +\infty}\frac n{\mathbb E[\mathcal R_n]}
    \mathbb P(T_0>n)\le
       \frac {1-x^\gamma }{1-x}$$
from which (\ref{gammaa2}) is deduced by letting the variables $x$ and $y$ converge to one.

If $\mathbb P(Z_1\in\{-1,0,1\})=1$ and if the distribution of $Z=(Z_n)_n$ is symmetric, then,
\begin{eqnarray*}
\mathbb P\left(Z_n^*\le -1\right)
&=&\mathbb P(T_0>n,\ Z_1<0)\\
&=&\mathbb P(T_0>n,\ Z_1>0)\\
&=&\frac 12 \mathbb P(T_0>n).
\end{eqnarray*}
and so
relation (\ref{EEE3ba}) directly follows from (\ref{gammaa2}).

\end{proof}

One can also obtain results as in Theorem~\ref{THMZ} even if $(\EE[Z_n^*])_n$ is not necessarily regularly varying:
\begin{prop}
Let $(Z_n)_{n\ge 0}$ be a
$\mathbb Z$-valued centered process with stationary increments. If there exists an $\varepsilon>0$ such that 
$$
\rho:=\liminf_{n\rightarrow +\infty}\frac{\mathbb E[\mathcal R_{n+\lfloor \varepsilon n\rfloor}]}{\mathbb E[\mathcal R_n]}>1
$$ 
then
$$ 
\frac{\rho-1}\varepsilon\leq\liminf_{n\rightarrow +\infty}\frac n{\mathbb E[\mathcal R_n]}\mathbb P(T_0>n)\leq \limsup_{n\rightarrow +\infty}\frac n{\mathbb E[\mathcal R_n]}\mathbb P(T_0>n)\le 1.
$$ 
\end{prop}

\begin{proof}
It suffices to put $y=1+\varepsilon$ and $x=0$ in (\ref{Eqxy}). 
\end{proof}


\begin{proof}[Proof of Theorem~\ref{thm:lower}]
{\it Step 1: Basic estimate.}  Fix $\varepsilon>0$ as in the assumption. Since the increments of $Z$ are stationary and $Z_0=0$, we can deduce that 
\begin{equation}\label{changetime}\mathbb P(Z_n^*<0)
=\mathbb P(Z_{k+1,n+k}^*<Z_k)\, .
\end{equation}
Let $M_n:=\#\{k=1,\ldots,\lfloor \varepsilon n\rfloor  \ :\   
         Z_{k+1,n+\lfloor \varepsilon n\rfloor}^*<Z_k\}$.
Observe that $M_n$ represents the number of records of the reversed process 
$(Z_{n+\lfloor\varepsilon n\rfloor-k}-Z_{n+\lfloor \varepsilon n\rfloor})_{k=0,...,n+\lfloor \varepsilon n\rfloor}$ 
occuring during the time interval
$\{n,...,n+\lfloor\varepsilon n\rfloor-1\}$.
Note that
\begin{eqnarray}
\lfloor \varepsilon n\rfloor\, \mathbb P(Z_{n}^*<0)
&\ge&
\sum_{k=1}^{\lfloor \varepsilon n\rfloor}\mathbb P\left( Z_{\lfloor (1+\varepsilon)n\rfloor-k}^*<0\right)\nonumber
\\
&=& \sum_{k=1}^{\lfloor \varepsilon n\rfloor}\mathbb P\left(Z_{k+1,\lfloor (1+\varepsilon)n\rfloor}^*<Z_k\right)\nonumber\\
&=&\mathbb E\left[ M_n\right]\, ,\label{EQUA1}
\end{eqnarray}
where we used \eqref{changetime}.
We set $R_{M_n+1}:=\inf\{k=\lfloor \varepsilon n\rfloor+1,\ldots,n+\lfloor \varepsilon n\rfloor
    \ :\ Z_k>Z^*_{k+1,n+\lfloor \varepsilon n\rfloor}\}$
with the convention $Z^*_{n+\lfloor \varepsilon n\rfloor+1,n+\lfloor \varepsilon n\rfloor}=-\infty$
(so that $R_{M_n+1}$ is well-defined and smaller than or equal to $n+\lfloor \varepsilon n\rfloor$)
 and
\begin{eqnarray*}
\forall i=1,\ldots,M_n :\quad R_i&:=&\sup\{k=1,\ldots,R_{i+1}-1\ :\
Z_k>Z_{k+1,n+\lfloor \varepsilon n\rfloor}^*
\}\\
&=&\sup\{k=1,...,R_{i+1}-1\, :\, Z_k>Z_{R_{i+1}}\}\, .
\end{eqnarray*}
Note that $R_i$ is the $i$-th smallest integer $k$ such that $Z_k$
is larger than all the following values $Z_{k+1},...,Z_{n+\lfloor \varepsilon n\rfloor}$.
Note that $Z_{R_1}=Z_{n+\lfloor \varepsilon n\rfloor}^*$ and $Z_{R_{M_n+1}}
=Z^*_{\lfloor \varepsilon n\rfloor+1,n+\lfloor \varepsilon n\rfloor}$.
Hence
\begin{eqnarray}
Z_{n+\lfloor \varepsilon n\rfloor}^*
-Z^*_{\lfloor \varepsilon n\rfloor+1,n+\lfloor \varepsilon n\rfloor} 
&=&  Z_{R_1}-Z_{R_{M_n+1}}
   =\sum_{i=1}^{M_n} (Z_{R_i}-Z_{R_{i+1}})\nonumber\\
&\le&  \sum_{i=1}^{M_n} (Z_{R_i}-Z_{R_{i}+1})\nonumber\\
&\le&
\left(\sup_{k=1,\ldots,\lfloor \varepsilon n\rfloor}   (Z_{k}-Z_{k+1})\right)M_n\, ,\label{MinoMn}
\end{eqnarray}
where we used the fact that $Z_{R_{i+1}}\ge Z_{R_i+1}$, which comes from the definition of $R_i$.
Observe that \eqref{MinoMn} becomes $0\le 0$ if $M_n=0$.
Combining
\eqref{EQUA1} and \eqref{MinoMn}, we obtain that, for every $b>0$,
\begin{eqnarray}
\lfloor \varepsilon n\rfloor\mathbb P(Z_n^*<0)&\ge&
\mathbb E[M_n]\nonumber\\
&\ge& \mathbb E\left[M_n
       \mathbf   1_{\{\sup_{k=1,\ldots,\lfloor
         \varepsilon n\rfloor}   (Z_{k}-Z_{k+1})\le b\}}\right]\nonumber\\
&\ge&  b^{-1}\, \mathbb E\left[(Z_{n+\lfloor \varepsilon n\rfloor}^*  -Z^*_{\lfloor\varepsilon n\rfloor+1,n+\lfloor \varepsilon n\rfloor})\mathbf   1_{\{\sup_{k=1,\ldots,\lfloor
         \varepsilon n\rfloor}   (Z_{k}-Z_{k+1})\le b\}}\right]\, .\label{keyinequality}
\end{eqnarray}
Moreover
\begin{equation}\label{Esup}
\mathbb E[Z^*_{\lfloor\varepsilon n\rfloor+1,n+\lfloor \varepsilon n\rfloor}]=
\mathbb E[Z_{\lfloor\varepsilon n\rfloor}]+\mathbb E[Z^*_{\lfloor\varepsilon n\rfloor+1,n+\lfloor \varepsilon n\rfloor}-Z_{\lfloor\varepsilon n\rfloor}]=0+
\mathbb E[Z_n^*]=\mathbb E[Z_n^*]\, .
\end{equation}
{\it Step 2: Proof of (i).} 
If $\mathbb P(-Z_1\le b)=1$, due to \eqref{keyinequality} and to \eqref{Esup},
we obtain
\[
\lfloor \varepsilon n\rfloor\mathbb P(Z_n^*<0)\ge \mathbb E[M_n]
\ge  b^{-1}\mathbb E\left[Z_{n+\lfloor \varepsilon n\rfloor}^*-Z^*_{n}\right]
\]
and so
$$\liminf_{n\rightarrow +\infty}\frac{\varepsilon n}{\mathbb E[Z_n^*]}\mathbb P(Z_n^*<0)\ge \frac{\rho-1}b .$$

{\it Step 3: Proof of (ii).
} Due to \eqref{keyinequality}, For any $b>0$,
\begin{eqnarray}
\lfloor \varepsilon n\rfloor\mathbb P(Z_n^*<0)
&\ge&
 b^{-1}\, \mathbb E\left[((Z_{n+\lfloor \varepsilon n\rfloor}^*     -Z_{n+\lfloor \varepsilon n\rfloor})-(Z^*_{\lfloor\varepsilon n\rfloor+1,n+\lfloor \varepsilon n\rfloor}- Z_{n+\lfloor \varepsilon n\rfloor}))\mathbf   1_{\{\sup_{k=1,\ldots,\lfloor
         \varepsilon n\rfloor}   (Z_{k}-Z_{k+1})\le b\}}\right]\nonumber\, .
\end{eqnarray}
Since $Z^*_{\lfloor\varepsilon n\rfloor+1,n+\lfloor \varepsilon n\rfloor}- Z_{n+\lfloor \varepsilon n\rfloor}\ge 0$, we have
\begin{eqnarray*}
\mathbb E\left[\left(Z^*_{\lfloor\varepsilon n\rfloor+1,n+\lfloor \varepsilon n\rfloor}-Z_{n+\lfloor \varepsilon n\rfloor}\right)
\mathbf   1_{\{\sup_{k=1,\ldots,\lfloor
         \varepsilon n\rfloor}   (Z_{k}-Z_{k+1})\le b\}}\right]
&\le&\mathbb E[Z^*_{\lfloor\varepsilon n\rfloor+1,n+\lfloor \varepsilon n\rfloor}-Z_{n+\lfloor \varepsilon n\rfloor}]\\
&=&\mathbb E[Z^*_{\lfloor\varepsilon n\rfloor+1,n+\lfloor \varepsilon n\rfloor}-Z_{\lfloor\varepsilon n\rfloor}]=\mathbb E[Z^*_{n}]\, ,
\end{eqnarray*}
where we used the facts that $\mathbb E[Z_{n+\lfloor \varepsilon n\rfloor}]=\mathbb E[Z_{\lfloor\varepsilon n\rfloor}]=0$ and that $Z^*_{\lfloor\varepsilon n\rfloor+1,n+\lfloor \varepsilon n\rfloor}-Z_{\lfloor\varepsilon n\rfloor}\stackrel{\mathcal L}=Z_n^*$.
Therefore, using again the fact that $\mathbb E[Z_{n+\lfloor \varepsilon n\rfloor}]=0$, we obtain
\begin{eqnarray*}
&\ &\lfloor \varepsilon n\rfloor\mathbb P(Z_n^*<0)\\
&\ge&
  b^{-1}\left( \mathbb E\left[(Z_{n+\lfloor \varepsilon n\rfloor}^*-Z_{n+\lfloor \varepsilon n\rfloor})
\mathbf   1_{\{\sup_{k=1,\ldots,\lfloor
         \varepsilon n\rfloor}   (Z_{k}-Z_{k+1})\le b\}}\right] -
   \mathbb E[Z^*_{n}]\right)\, \\
&\ge&  b^{-1}\left( \mathbb E\left[Z_{n+\lfloor \varepsilon n\rfloor}^*-Z^*_{n}\right]-
\mathbb E\left[(Z_{n+\lfloor \varepsilon n\rfloor}^*-Z_{n+\lfloor \varepsilon n\rfloor}) 
\mathbf   1_{\{\sup_{k=1,\ldots,\lfloor
         \varepsilon n\rfloor}   (Z_{k}-Z_{k+1})> b\}}\right]\right)
\nonumber\\
 &\ge&   b^{-1}\left( \mathbb E\left[Z_{n+\lfloor \varepsilon n\rfloor}^*-Z^*_{n}\right]-\Vert Z_{n+\lfloor \varepsilon n\rfloor}^*-Z_{n+\lfloor \varepsilon n\rfloor}\Vert_p\left(
       \mathbb P\left(\sup_{k=1,\ldots,\lfloor
         \varepsilon n\rfloor}   (Z_{k}-Z_{k+1})> b\right)\right)^{\frac 1q}\right)\label{estimHolder}\\
&\ge&b^{-1}\left( \mathbb E\left[Z_{n+\lfloor \varepsilon n\rfloor}^*-Z^*_{n}\right]-\Vert Z_{n+\lfloor \varepsilon n\rfloor}^*-Z_{n+\lfloor \varepsilon n\rfloor}\Vert_p\left(
       \varepsilon n\mathbb P(-Z_1> b)\right)^{\frac 1q}\right)\, ,\nonumber
\end{eqnarray*}
with $q$ such that $\frac 1p+\frac 1q=1$.
Let $d>0$ and $K:=\limsup_{n\to+\infty} n \mathbb P(-Z_1>b_{n})$.
We obtain
$$
\liminf_{n\rightarrow +\infty} \frac{ \varepsilon n\, b_{ dn}}{\mathbb E[Z_n^*]}\mathbb P(Z_{n}^*<0)    \ge \rho - \kappa  (\varepsilon K/d)^{1/q} - 1,
$$
which can be made positive by an appropriate choice of $d$, since $\rho>1$.\\

{\it Step 4: Proof of (iii).} Due to \eqref{keyinequality},
$$
\lfloor \varepsilon n\rfloor\mathbb P(Z_n^*<0)\ge  b_{dn}^{-1}\, \mathbb E\left[(\max(0,Z_{n+\lfloor \varepsilon n\rfloor}^*)  -|Z_1|-Z^*_{\lfloor\varepsilon n\rfloor+1,n+\lfloor \varepsilon n\rfloor})\mathbf   1_{\{\sup_{k=1,\ldots,\lfloor
         \varepsilon n\rfloor}   (Z_{k}-Z_{k+1})\le b_{dn}\}}\right]\, .
$$
Moreover
$$\mathbb E\left[(|Z_1|+Z^*_{\lfloor\varepsilon n\rfloor+1,n+\lfloor \varepsilon n\rfloor})\mathbf   1_{\{\sup_{k=1,\ldots,\lfloor
         \varepsilon n\rfloor}   (Z_{k}-Z_{k+1})\le b\}}\right]\\
\le\mathbb E\left[\max\left(0,Z^*_{\lfloor\varepsilon n\rfloor+1,n+\lfloor \varepsilon n\rfloor}\right)+|Z_1|)\right]\, .$$
Therefore
\begin{eqnarray*}
&\ &\lfloor \varepsilon n\rfloor\mathbb P(Z_n^*<0)\\
&\ge& b_{ dn}^{-1}\left( \mathbb E\left[\max(0,Z_{n+\lfloor \varepsilon n\rfloor}^*)\mathbf 1_{\{\sup_{k=1,\ldots,\lfloor \varepsilon n\rfloor}  
                         (Z_{k}-Z_{k+1})\le b_{ dn}\}}\right]-
   \mathbb E\left[\max\left(0,Z^*_{\lfloor\varepsilon n\rfloor+1,n+\lfloor \varepsilon n\rfloor}\right)+|Z_1|\right]\right)\\
&\ge& b_{ dn}^{-1}\left( \mathbb E\left[\mathbb  E[\max(0,Z_{n+\lfloor \varepsilon n\rfloor}^*)|\mathcal F_0] \mathbb P\left(\left.\sup_{k=1,\ldots,\lfloor \varepsilon n\rfloor}  
                         (Z_{k}-Z_{k+1})\le b_{dn}\right|\mathcal F_0\right)\right]\right.\\
&\ &\ \ \ \ \ \ \ \  \left.-
   \mathbb E\left[\max\left(0,Z^*_{\lfloor\varepsilon n\rfloor+1,n+\lfloor \varepsilon n\rfloor}\right)+|Z_1|\right]\right)\nonumber\\
&\ge&  b_{ dn}^{-1}\left(\mathbb E\left[\mathbb  E[\max(0,Z_{n+\lfloor \varepsilon n\rfloor}^*)|\mathcal F_0] (1-\varepsilon n\mathbb P(-Z_1 > b_{ dn}))\right] -
\mathbb E\left[\max\left(0,Z^*_{\lfloor\varepsilon n\rfloor+1,n+\lfloor \varepsilon n\rfloor}\right)+|Z_1|\right]\right)\nonumber\\
&=&  b_{ dn}^{-1}\left( \mathbb E\left[\max(0,Z_{n+\lfloor \varepsilon n\rfloor}^*)\right]
           (1-\varepsilon n\mathbb P(-Z_1> b_{dn})) -  \mathbb E\left[\max\left(0,Z^*_{\lfloor\varepsilon n\rfloor+1,n+\lfloor \varepsilon n\rfloor}\right)+|Z_1|\right]\right)\\
&\ge&  b_{ dn}^{-1}\left( \mathbb E\left[Z^*_{n+\lfloor \varepsilon n\rfloor}\right]
           (1-\varepsilon n\mathbb P(-Z_1> b_{dn})) -  \mathbb E\left[\max\left(0,Z^*_{\lfloor\varepsilon n\rfloor+1,n+\lfloor \varepsilon n\rfloor}\right)+|Z_1|\right]\right)\, ,\label{EQUA2}
\end{eqnarray*}
where we used the fact that the increments of $Z$ are positively associated
conditionally to $\mathcal F_0$, and that their common conditional distribution is 
independent of $\mathcal F_0$.
This leads us to
\begin{eqnarray*}
\liminf_{n\to+\infty} \frac{\varepsilon n b_{\lfloor dn\rfloor}}{\EE[ Z_n^*]} \mathbb P( Z_n^* < 0)&\geq&
\rho(1-\varepsilon \frac K d) -\limsup_{n\rightarrow +\infty}\frac{\mathbb E\left[\max\left(0,Z^*_{\lfloor\varepsilon n\rfloor+1,n+\lfloor \varepsilon n\rfloor}\right)\right]}{\EE[ Z_n^*]}\, ,
\end{eqnarray*}
since \eqref{eqn:notregvar} implies that $\lim_{n\rightarrow +\infty}\mathbb E[Z_n^*]=+\infty$. We conclude by adjusting $d$.\\
Finally, let us prove the sufficient condition. To this end we observe that
\begin{eqnarray*}
\mathbb E\left[\max\left(0,Z^*_{\lfloor\varepsilon n\rfloor+1,n+\lfloor \varepsilon n\rfloor}\right)\right]&=&\EE\left[Z^*_{\lfloor\varepsilon n\rfloor+1,n+\lfloor \varepsilon n\rfloor}\right]-\mathbb E\left[\min\left(0,Z^*_{\lfloor\varepsilon n\rfloor+1,n+\lfloor \varepsilon n\rfloor}\right)\right]\\
&\le&\EE\left[Z^*_{\lfloor\varepsilon n\rfloor+1,n+\lfloor \varepsilon n\rfloor}-Z_{\lfloor\varepsilon n\rfloor}\right]-\mathbb E\left[\min\left(0,Z_{\lfloor\varepsilon n\rfloor+1}\right)\right]\\
&\le&\EE\left[Z^*_{n}\right]-\mathbb E\left[\min\left(0,Z_{\lfloor\varepsilon n\rfloor+1}\right)\right]\\
&\le&\EE\left[Z^*_{n}\right]+\mathbb E\left[\max\left(0,Z_{\lfloor\varepsilon n\rfloor+1}\right)\right]\, ,
\end{eqnarray*} 
since $\EE[Z_{\lfloor\varepsilon n\rfloor}]=0$ and since $Z^*_{\lfloor\varepsilon n\rfloor+1,n+\lfloor \varepsilon n\rfloor}-Z_{\lfloor\varepsilon n\rfloor}$ has the same distribution as $Z_n^*$, and using the fact that $Z^*_{\lfloor\varepsilon n\rfloor+1,n+\lfloor \varepsilon n\rfloor}\ge Z_{\lfloor\varepsilon n\rfloor+1}$.
\end{proof}

\begin{proof}[Proof of Lemma \ref{lem:ConditionAAA}]
Observe that
$$\lim_{n\rightarrow +\infty}\frac{\mathbb E\left[Z^*_{n+\lfloor \varepsilon n\rfloor}-\max\left(0,Z^*_{\lfloor\varepsilon n\rfloor+1,n+\lfloor \varepsilon n\rfloor}\right)\right]}{\mathbb E[Z_n^*]}
=\frac{\EE[\sup_{[0,1+\varepsilon]}\Delta-
     \max(0,\sup_{[\varepsilon,1+\varepsilon]}\Delta)]}{\EE[\sup_{[0,1]}\Delta]}\, $$
and that the strict positivity of this limit will ensure \eqref{AAA}. 
In particular, if there exists an $\varepsilon$ such that 
\begin{equation} \label{eqn:AAAA}
\mathbb P\left(\sup_{[\varepsilon,1+\varepsilon]}\Delta<0
,\ \sup_{[0,1+\varepsilon]}\Delta>0
\right)>0\, ,
\end{equation} 
then \eqref{AAA} holds true.
In order to
find such an $\varepsilon$, we proceed as follows.
Let $p:=\mathbb P\left( \sup_{[0,1+\varepsilon]}\Delta=0 \right)$ and $ p':=\mathbb P\left( \Delta_1<0 \right)$.
We know that $ p< p'$, so that $1-\frac { p'- p}2\in(0,1)$.
Choose $a$ such that $\mathbb P(\sup_{[0,1]}\Delta<a)>1-\frac { p'- p}2$, i.e.
$$\mathbb P\left(\sup_{[0,1]}(\Delta_{\varepsilon+\cdot}-\Delta_{\varepsilon})<a\right)>1-\frac { p'- p}2  \, .$$
Now fix $\varepsilon$
such that $ \mathbb P(\Delta_\varepsilon<-a)=\mathbb P(\varepsilon^{\alpha}\Delta_1<-a)>\frac { p'+ p}2$, so that
\begin{eqnarray*}
 p&=&1-\frac { p'- p}2+ \frac { p'+ p}2-1<\mathbb P\left(\sup_{[0,1]}(\Delta_{\varepsilon+\cdot}-\Delta_{\varepsilon})<a\right)+\mathbb P(\Delta_\varepsilon<-a)-1\\
&\le& \mathbb P\left(\sup_{[0,1]}(\Delta_{\varepsilon+\cdot}-\Delta_{\varepsilon})<a\right)+\mathbb P(\Delta_\varepsilon<-a)-
\mathbb P\left(\sup_{[0,1]}(\Delta_{\varepsilon+\cdot}-\Delta_{\varepsilon})<a\ \mbox{or}\ \Delta_\varepsilon<-a\right)\\
&\le &\mathbb P\left(\Delta_\varepsilon<-a,\ \sup_{[0,1]}(\Delta_{\varepsilon+\cdot}-\Delta_{\varepsilon})<a 
       \right)\le  \mathbb P\left(\sup_{[\varepsilon,1+\varepsilon]}\Delta<0\right)\, ,
\end{eqnarray*}
which implies (\ref{eqn:AAAA}) and thus \eqref{AAA}.
\end{proof}

\begin{rqe}\label{Condcompl} It follows from our proof that, in Theorem \ref{thm:lower}-Case (ii), Conditions \eqref{eqn:hypbn} and \eqref{Hypkappa} can be replaced by
the following condition
\begin{equation}\label{Hypcompliquee}
\limsup_{n\rightarrow +\infty} (\mathbb E[Z_n^*])^{-1}\mathbb E\left[(Z_{n+\lfloor \varepsilon n\rfloor}^*-Z_{n+\lfloor \varepsilon n\rfloor}) 
\mathbf   1_{\{\sup_{k=1,\ldots,\lfloor
         \varepsilon n\rfloor}   (Z_{k}-Z_{k+1})> b_{dn}\}}\right]<\rho-1,
\end{equation}
\end{rqe}
\begin{proof}[Proof of Corollary \ref{thmbounded}]
We apply (i) of Theorem \ref{thm:lower} with any $\varepsilon>0$
and $\rho:=(1+\varepsilon)^\gamma$ and let $\varepsilon\rightarrow 0$.
\end{proof}

%
%
%
%
%
%
%
%
\section{Sums of stationary sequences and fractional Brownian motion}  \label{sec:resultssumsof}
Let $(X_i)_{i\geq 0}$ be a stationary centered Gaussian sequence with variance $1$ and correlations 
$r(j):=\EE [X_0 X_j] = \EE[ X_k X_{j+k} ]$ satisfying as $n\rightarrow +\infty$,
\begin{equation} \label{eqn:lrd}
\sum_{i,j=1}^n r(i-j) = n^{2H} \ell(n), 
\end{equation}
where $H\in (0,1)$ and $\ell$ is a slowly varying function at infinity. We are interested in the persistence probabilities of 
$\left(Z_n:=\sum_{i=1}^n X_i\right)_{n\ge 0}$ (with the usual convention $Z_0:=0$).
We recall that the scaling limit of $(Z_n)_ {n\ge 0}$ is the fractional Brownian motion $B_H$ with Hurst parameter $H$, (see \cite{taqqu}, \cite[Theorem 4.6.1]{Whitt}):
\begin{equation} \label{eqn:weakconvergencetofbm}
\left( n^{-H} \ell(n)^{-1/2} Z_{[ nt ]}\right)_{t\geq 0}   \mathop{\Longrightarrow}_{n\rightarrow\infty}^{\mathcal{L}} \left( B_H(t)\right)_{t\geq 0},
\end{equation}
and that $B_H$ is a real centered Gaussian process with
covariance function
\[
\mathbb E[B_H(t)B_H(s)]=\frac 12(t^{2H}+s^{2H}-|t-s|^{2H}).
\]
A sequence satisfying relation (\ref{eqn:lrd}) is said to have {\it long-range dependence} if $H>1/2$. We refer to \cite{samorodnitsky} for a recent overview of the field.

In this setup, we obtain the following theorem.
\begin{thm} \label{thm:statseq}
Assume $(X_i)_{i\ge 0}$ is a stationary centered Gaussian sequence such that (\ref{eqn:lrd}) holds with $H\in (0,1)$ and $\ell$ slowly varying. 
Then, for every $a>0$ there is some constant $c>0$ such that, for every $n\geq 1$,
\begin{equation} \label{eqn:thmlrd1}
c^{-1}\ n^{-(1-H)} \frac{\sqrt{\ell(n)}}{\sqrt{\log n}}\leq
\PP\left(Z_n^* <0 \right)\mbox{ and }
\PP\left( Z_n^* \leq -a \right) \leq  c\ n^{-(1-H)} \sqrt{\ell(n)}.
\end{equation}
If moreover, 
$\inf_{n\ge 1} \sum_{i=1}^n r(i-1)=\inf_{n\ge 1}\mathbb E[Z_1Z_n]> 0$,
then, for every $b\in\mathbb R$, there is some constant $c>0$ such that
\begin{equation} \label{eqn:thmlrd2}
\forall n\ge 1~:\quad  c^{-1} n^{-(1-H)} \frac{\sqrt{\ell(n)}}{\sqrt{\log n}}e^{-c\sqrt{\log n}}\leq \PP\left( Z_n^* \leq b \right) \leq   n^{-(1-H)} {\sqrt{\ell(n)}}e^{c\sqrt{\log n}}.
\end{equation}
If moreover the correlation function $r$ is non-negative (which implies that $H\ge 1/2$) then, for every $b\in\mathbb R$, there is some constant $c>0$ such that
\begin{equation} \label{eqn:thmlrd3}
\forall n\ge 1~:\quad  c^{-1} n^{-(1-H)} \frac{\sqrt{\ell(n)}}{\sqrt{\log n}}\leq \PP\left(Z_n^* \leq b \right) \leq  c\ n^{-(1-H)} \sqrt{\ell(n)}.
\end{equation}
\end{thm}
Note that, for negative $b$, the second inequality in \eqref{eqn:thmlrd1} provides a more precise upper bound than
the one of \eqref{eqn:thmlrd2} and that, for positive $b$,
the first inequality in \eqref{eqn:thmlrd1} provides a more precise 
lower bound that the one of \eqref{eqn:thmlrd2}.
\begin{proof}[Proof of Theorem~\ref{thm:statseq}]

Let $\sigma_n := \sup_{k=1,\ldots,n}\Vert Z_k\Vert_2$.
We set $a_n : = \Vert Z_n\Vert_2 = n^{H}\sqrt{\ell(n)}$.
Due to Karamata's characterization of slowly varying functions \cite{Karamata}, there exist two functions $\varepsilon,L_0$
such that $L_0(\infty):=\lim_{t\rightarrow +\infty}L_0(t)$ exists in $(0,+\infty)$
and $\lim_{t\rightarrow +\infty}\varepsilon(t)=0$ and such that
$$\forall n\in\mathbb N^*,\quad a_n:= L_0(n)\, n^{H}e^{\int_1^n\frac{\varepsilon(t)}t\, dt}\, .$$
We first prove that $\sigma_n\sim a_n$ as $n\rightarrow +\infty$.
Since $a_n\le\sigma_n$, it is enough to prove that $\limsup_{n\rightarrow +\infty}\frac{\sigma_n}{a_n}\le 1$.
Given $\varepsilon_0\in(0,H)$, there exists an integer $k_0\ge 1$
such that
$$ \forall t\ge k_0,\quad |\varepsilon(t)|\le\varepsilon_0\quad\mbox{and}\quad  (1-\varepsilon_0)L_0(\infty)\le L_0(t)\le (1+\varepsilon_0)L_0(\infty)\, .$$
For this choice of $k_0$ and for every $n\ge k_0$, we have
\[
\sup_{k=k_0,...,n}\frac{a_k}{a_n}\le
    \sup_{k=k_0,...,n}\frac{L_0(k)}{L_0(n)}\left(\frac kn\right)^He^{\int_k^n\frac{\varepsilon_0}t\, dt}\\
\le\frac{1+\varepsilon_0}{1-\varepsilon_0}\sup_{k=k_0,...,n}\left(\frac kn\right)^{H-\varepsilon_0}
\]
and so
$$\limsup_{n\rightarrow +\infty}\frac{\sigma_n}{a_n}
   \le\limsup_{n\rightarrow +\infty}\left(\frac{\max_{k=1,...,k_0}a_k}{a_n}+ \frac{1+\varepsilon_0}{1-\varepsilon_0}\right)=\frac{1+\varepsilon_0}{1-\varepsilon_0}\, .$$
This inequality being true for every $\varepsilon_0>0$, we conclude
that  $\limsup_{n\rightarrow +\infty}\frac{\sigma_n}{a_n}\le 1$. and so that $\sigma_n\sim a_n$.\\
Note that, setting $b_n:= n^{H}e^{\int_1^n\frac{\varepsilon(t)}t\, dt}$, there exists a $c_1>0$ such that, for every positive integer $n$,
$ a_n\le\sigma_n\le c_1b_n$ and such that $c_1^{-1}b_n\le a_n$ for every $n$ large enough so that $a_n>0$. Let us denote by $D_n$ the Dudley integral
\begin{equation}\label{Dudley}
D_n :=  \int_0^{\sigma_n/2} \sqrt{\log N(n,t)} \dd t,
\end{equation}
where $N(n,t)$ is the smallest number of closed balls in $\{0,\ldots,n\}$ of radius $t$ for the pseudo-metric $d(k,\ell)=a_{|k-\ell|}$ which form
a covering of $\{0,\ldots,n\}$.
Note that if $c_1b_r\leq t$ for some $r\in\{1,...,n-1\}$ then, for every integer $k\ge 0$, the set $\{k,.., k+2r\}$ is contained in a closed ball of center $k+r$ and of radius $\sigma_r\le t$, therefore $N(n,t)\leq
\left\lceil \frac{n+1}{2r+1}\right\rceil\le \lceil n/r\rceil \le n/r+1\le 2n/r$. Further, trivially $N(n,t)\leq n+1$ for any $t$ and $N(n,t)=1$ for $t>c_1b_n\ge \sigma_n$. 
Therefore, for $\vartheta\in(0,1)$, for $n$ large enough,
\begin{eqnarray*}
D_n &\leq& \int_0^{b_n/\sqrt{\log (n+1)}} \sqrt{\log (n+1)} \,\dd t + \int_{b_n/\sqrt{\log (n+1)}}^{c_1b_n} \sqrt{\log N(n,t)} \,\dd t\\
&\leq& b_n +\sum_{k=n^{\vartheta}}^{n} 
     \int_{c_1 b_{k-1}}^{c_1b_{k}} \sqrt{\log (2n/(k-1))}\, dt\\
&=& b_n +\sum_{k=n^{\vartheta}}^{n} 
       c_1(b_k-b_{k-1}) \sqrt{\log (2n/(k-1))}.
\end{eqnarray*}
But, due to the form of $b_k$,
$\frac{b_k-b_{k-1}}{b_k}
=\frac{k^H-(k-1)^H}{k^H}+\frac{(k-1)^H\left(1-e
        ^{\int_{k}^{k-1}\frac{\varepsilon(t)}t\, dt}\right)}{k^H}=O(k^{-1})$. Let $\upsilon\in(0,H)$. Therefore there exists a $c_2>0$
such that
\begin{eqnarray}
D_n &\leq& b_n +c_2b_n\frac 1n\sum_{k=n^{\vartheta}}^{n} 
       \frac {n b_k}{k b_n} \sqrt{\log (2n/(k-1))}\nonumber\\
&\leq& b_n +c_2b_n\frac 1n\sum_{k=n^{\vartheta}}^{n} 
       (k/n)^{H-1-\upsilon}=O(a_n)\, ,\label{majoEn}
\end{eqnarray}
for $n$ large enough (where we used again the expression of $b_n$).
Due to \cite[Corollary 2, p.181]{Lifshits},
$$
\forall u>0~:\quad \P \Big(  Z_n^* > u+4\sqrt{2}D_n\Big)  \le 
         \mathbb P(Z_1>u/\sigma_n)\, .
$$
Therefore,
\begin{eqnarray*}
\EE\left[ \sup_{k=1,\ldots, n} Z_k^2\right] 
&=&\int_0^{+\infty}  \P \Big(\sup_{k=1,\ldots, n} |Z_k|^2  > u\Big)\, \dd u\\
&\leq& 32 D_n^2   +  \int_{32 D_n^2}^\infty  \P \Big(\sup_{k=1,\ldots, n} |Z_k|  >\sqrt{u}\Big)\, \dd u\\
&\leq& 32 D_n^2   +  \int_{32 D_n^2}^\infty  \left(\P \Big(\sup_{k=1,\ldots, n} Z_k  >\sqrt{u}\Big)+\P \Big(\sup_{k=1,\ldots, n} (-Z_k)  >\sqrt{u}\Big)\right)\, \dd u\\
  &\le& 32D_n^2 +   2\int_{4\sqrt{2} D_n}^\infty
 \P \Big( Z_n^* >x\Big) 2x\,  \dd x\\
  &=& 32D_n^2 +   \int_{0}^\infty
 \P \left(Z_1>u/\sigma_n\right) 4(4\sqrt{2}D_n+u)\,  \dd u\, \\
  &=&32D_n^2 +   \sigma_n\int_{0}^\infty
 \P \left(Z_1>v\right) 4(4\sqrt{2}D_n+v \sigma_n) \, \dd v\, \\
&\leq &32D_n^2 +   C(\sigma_nD_n+ \sigma_n^2)=O(a_n^2)\, ,
\end{eqnarray*}
due to \eqref{majoEn} and to the fact that $\sigma_n\sim a_n$.
Thus, we have proved that 
$\Vert Z_n^*\Vert_2=O(n^{H} \sqrt{\ell(n)})$
and thus the uniform integrability of 
$(Z_n^*/ (n^H \sqrt{\ell(n)}))_n$,
which combined with \eqref{eqn:weakconvergencetofbm} yields
$$
\frac{1}{n^H \ell(n)^{1/2}}\,\EE\left[  Z_n^* \right] \to \EE \left[\sup_{t\in[0,1]} B_H(t)\right] \in (0,\infty).
$$
Hence, Theorems \ref{THMlem} and \ref{thm:lower}-(ii) with $p=2$ 
(observing that $Z$ is reversible and thus that $Z^*_n-Z_n$ has the same distribution as $Z_n^*$)
and Remark \ref{MAINTHM} imply
both \eqref{eqn:thmlrd1} and \eqref{eqn:thmlrd3}
(due to Lemma~\ref{lem:positiveassociation} for \eqref{eqn:thmlrd3} since in this case the increments of $Z$ 
are positively correlated standard Gaussians and thus are  positively associated and have the same distributions). The fact that the lower bound holds true
for every $n$ (up to changing the constant $c$) comes from the fact that
$\mathbb P(Z_n^*<m)$ is decreasing in $n$ and thus cannot be 0 due to the asymptotic lower bound.

It remains to prove \eqref{eqn:thmlrd2}. 
We are now going to use Proposition~1.6 in \cite{aurzadadereich}. For this purpose, let us denote by $\mathcal H$ the reproducing kernel Hilbert space of the process $(Z_k)$, i.e. the Hilbert space made of elements of $\mathbb R^\mathbb N$, generated
 by $\{K(n,\cdot),\ n\in\mathbb N\}$ together with the scalar product
given by $\langle K(n,\cdot),K(m,\cdot)\rangle_{\mathcal H}=K(n,m)$, with $K(n,m):=\mathbb E[Z_nZ_m]$.
Let $a>0$.
Assume that 
$\kappa:=\inf_n \sum_{i=1}^n r(i-1)=\inf_n K(1,n)> 0$.
Consider the function $f(n):=2 aK(1,n)/\kappa \in \mathcal H$. 
Note that
$f(k)\geq 2a$, $k\geq 1$, and that
$$
\|f\|_{\mathcal H}^2 = \langle f,f\rangle_{\mathcal H} = \left(\frac{2a}{\kappa}\right)^2\langle K(1,.),K(1,.)\rangle_{\mathcal H}= (2a/\kappa)^2 K(1,1) = (2a/\kappa)^2,
$$
where we used $K(1,1)=\mathbb E[Z_1^2]=1$ by assumption. Using the last two properties and Proposition~1.6 in \cite{aurzadadereich} we can conclude that
\begin{align*}
\P(  Z_n^* \leq -a)
& = \P( \forall k\leq n \, :\,  Z_k + f(k) \leq -a + f(k) )
\\
& \geq \P( \forall k\leq n \, :\, Z_k + f(k) \leq a )
\\
& \geq \P( Z_n^* \leq a ) \, \exp( - \sqrt{2 \|f\|_{\mathcal H}^2 \log ( 1/\P(Z_n^* \leq a ))} - \|f\|_{\mathcal H}^2/2)\, .
\end{align*}
Inserting now the lower bound given by \eqref{eqn:thmlrd1} and the value for $\|f\|_{\mathcal H}$ we obtain, for every $\alpha>1-H$,
$$
\P(  Z_n^* \leq -a)
\geq \P(  Z_n^* \leq a)
 \, \exp( - a\sqrt{ 8 \alpha \kappa^{-2} \log n} - 2a^2/\kappa^2),
$$
for $n$ large enough.
This combined with \eqref{eqn:thmlrd1} gives \eqref{eqn:thmlrd2}. 
\end{proof}
A simple example in Theorem~\ref{thm:statseq} is $(Z_n=B_H(n))_n$. Theorem~\ref{thm:statseq} holds and gives
$$
\forall a>0~:\quad \PP\left( \max_{k=1,\ldots,n} B_H(k) \leq -a \right) \leq O( n^{-(1-H)}) .
$$
Let $a>0$. When $H\in (\frac{1}{2},1)$, due to Lemma~\ref{lem:positiveassociation}, we obtain
$$
\quad \PP\left( \max_{k=1,\ldots,n} B_H(k) \leq a \right) =O( n^{-(1-H)}),
$$
while for $H\in(0,\frac{1}{2})$ we get
$$
\quad \PP\left( \max_{k=1,\ldots,n} B_H(k) \leq a \right)=O\left( n^{-(1-H)} e^{c \sqrt{\log n}}\right).
$$
From these computations, we deduce an upper bound for the persistence probability of the continuous-time fractional Brownian motion because trivially:
$$
\mathbb P\left(\sup_{t\in[0,T]}B_H(t)\le a\right)\le \mathbb P\left(\max_{
k=1,\ldots,\lfloor T\rfloor
}B_H(k)\le a\right).
$$
We also get a lower bound, as detailed in the next result.
\begin{thm}[Fractional Brownian motion]\label{fBM}
Let $H\in(1/2,1)$. Let $B_H$ be a fractional Brownian motion
of Hurst parameter $H$. Then, for every $a>0$,
there exists a $c>0$ such that
\[
c^{-1} T^{-(1-H)}(\log T)^{-\frac 1{2H}}\le \mathbb P\left(\sup_{t\in[0,T]}
         B_H(t)\le a\right)=O( T^{-(1-H)}).
\]
The lower bound holds for any $H\in(0,1)$.
\end{thm}
This result improves some results of \cite{Molchan1999,Aurzada}
when $H>1/2$.
We remark that the upper bound also holds for $H<1/2$ with an additional logarithmic factor $(\log T)^{(2-H)/H+o(1)}$, by \cite{Aurzada}, but we do not improve this result here. 
\begin{proof}[Proof of Theorem \ref{fBM}]
It remains to prove the lower bound.
Due to \eqref{eqn:thmlrd1}, there exists a $c>0$ such that
$$
\mathbb P\left(\max_{k=1,\ldots,n}B_H(k)< 0\right)\ge c \frac{n^{-(1-H)}}{\sqrt{\log n}}.
$$
Let $C>0$.
Moreover,
$$\max_{t\in[0,n]}B_H(t)-\max_{k=0,\ldots,n\lfloor (C\log n)^{1/(2H)}\rfloor -1}B_H(n_{k+1})
      \le \max_{k=0,\ldots,n\lfloor (C\log n)^{1/(2H)}\rfloor-1}A_k,$$
with 
$$A_k:=\max_{s\in [n_k,n_{k+1}]}\left( B_H (s) - B_H\left(n_{k+1}\right)\right)$$
where $n_k:= \frac{k}{\lfloor (C\log n)^{1/(2H)}\rfloor}$.
Then, using the stationarity and the self-similarity of the fractional Brownian motion, 
\begin{eqnarray*}
&\ &\mathbb P\left(\max_{k=0,\ldots,n\lfloor (C\log n)^{1/(2H)}\rfloor-1}A_k>a\right)\le n(C\log n)^{1/(2H)}\mathbb P\left(\max_{t\in [0,1]}B_H(tn_1)-B_H(n_1)>a \right)\\
&\le& n(C\log n)^{1/(2H)}\mathbb P\left(\max_{t\in [0,1]}B_H(t)-B_H(1)>a\lfloor (C\log n)^{1/(2H)}\rfloor^{H}\right)\\
&\le& n(C\log n)^{1/(2H)}\mathbb P\left(\max_{t\in [0,1]}B_H(t)>a\lfloor (C\log n)^{1/(2H)}\rfloor^{H}\right)\\
&\le& n(C\log n)^{1/(2H)} e^{-ba^2  C\log n}
\end{eqnarray*}
for some constant $b>0$ (see Statement 2 in \cite{Molchan1999} for instance), 
where we used the reversibility of $B_H$ at the penultimate line. 
We choose $C=C(a)$ large enough so that this last quantity is in $o(n^{-(1-H)}(\log n)^{-\frac 1{2H}})$. So,
using the self-similarity of $B_H$, we obtain
\begin{eqnarray*}
&& \mathbb P\left(\max_{t\in[0,n]}B_H(t)\le a\right)\\
&\ge& \mathbb P\left(\max_{k=1,\ldots,n\lfloor(C\log n)^{1/(2H)}\rfloor}B_H\left(\frac{k}{\lfloor (C\log n)^{1/(2H)}\rfloor}\right)< 0\right)-o(n^{-(1-H)}(\log n)^{-\frac  1{2H}})\\
&\ge& \mathbb P\left(\lfloor (C\log n)^{1/2H}\rfloor^{-H}
        \max_{k=1,\ldots,n\lfloor (C\log n)^{1/(2H)}\rfloor}B_H(k)< 0\right)-o(n^{-(1-H)}(\log n)^{-\frac 1{2H}})\\
&\ge&\mathbb P\left(\max_{k=1,\ldots,n\lfloor(C\log n)^{1/(2H)}\rfloor}B_H(k)< 0\right)-o(n^{-(1-H)}(\log n)^{-\frac 1{2H}})\\
&\ge& \frac{c\, C^{\frac{H-1}{2H}}}
    {2}\, n^{-(1-H)}(\log n)^{-\frac 1{2H}},
\end{eqnarray*}
for every $n$ large enough, due to \eqref{eqn:thmlrd1}.
\end{proof}

\section{Random walks in random sceneries} \label{sec:rwrs}
Random walks in random sceneries were introduced independently by H. Kesten and F. Spitzer \cite{KS} and by A. N. Borodin \cite{Borodin}.
Let $d$ be a positive integer and $S = (S_n)_{n\ge 0}$ be a random walk in $\mathbb{Z}^d$ starting at $0$,
i.e., $S_0 = 0$ and
$
X_n:= S_n-S_{n-1}, n \ge 1$ is a sequence of i.i.d.\ $\mathbb{Z}^d$-valued random variables.
Let $\xi = (\xi_x)_{x \in \mathbb{Z}^d}$ be a field of i.i.d.\ one-dimensional real
valued random variables independent of $S$.
The field $\xi$ is called the random scenery.
{\it The random walk in random scenery (RWRS)} $Z := (Z_n)_{n \ge 0}$ is defined 
by setting $Z_0 := 0$ and, for every integer $n\ge 1$,
$Z_n := \sum_{i=1}^n \xi_{S_i}$. 
We will denote by $\mathbb{P}$ the joint law of $S$ and $\xi$.
Limit theorems for RWRS have a long history, we refer to \cite{GuPo} for a complete review. In particular, RWRS have 
stationary increments which are positively associated conditionally to the random walk $S$.

We assume without any loss of generality that the support of the distribution of $X_1$ is not contained in a proper subgroup of $\mathbb Z^d$ and that the closed subgroup generated by the support of
the distribution of $\xi_0$ is either $\mathbb Z$ or $\mathbb R$. We consider the case when the distribution of $\xi_0$ is centered and square integrable. Let $W=(W(x))_{x\in\mathbb R}$ be a Brownian motion of variance $\mathbb E[\xi_0^2]$.

When $S$ is  transient, $((n^{-\frac 12} Z_{[nt]})_{t\ge 0})_{n\ge 1}$ converges in distribution in $(D([0,+\infty)),J_1)$
to $(\Delta_t:=c_0W(t))_{t\geq 0}$ for some $c_0>0$ (see \cite{Borodin} and \cite{KS}).

When $S$ is recurrent (which may happen for $d=1$ and $d=2$ only), we assume moreover that the distribution of $X_1$ is in the normal domain of attraction of 
a stable distribution of index $\alpha\in[d,2]$ (if $d=2$, we assume that 
$(n^{-\frac 12} S_n)_{n\geq 1}$
converges in distribution
to a bidimensional gaussian random variable).
\\
If the walk $S$ is recurrent and $\alpha=d$, then
$(((n\log n)^{-\frac 12} Z_{[nt]})_{t\ge 0})_{ n\ge 1}$ converges in distribution  in 
$(D([0,+\infty)),J_1)$
to $(\Delta_t:=c_1W(t))_{t\geq 0}$ for some $c_1>0$ (see \cite{Bo89}).\\
If $S$ is recurrent and $d=1<\alpha$, the following convergence holds in 
$(\mathcal D([0,+\infty)),J_1)$:
$$\left(n^{-\frac{1}{\alpha}} S_{\lfloor nt\rfloor}\right)_{t\geq 0}   
\mathop{\Longrightarrow}_{n\rightarrow\infty}
^{\mathcal{L}} \left( Y(t)\right)_{t\geq 0},$$
where $Y$ is an $\alpha$-stable L\'evy process, which is assumed to be independent of $W$. In this case, in \cite{KS}, Kesten and Spitzer proved 
the following convergence in distribution in $(\mathcal D([0,+\infty)),J_1)$,
$$(n^{- 1+\frac 1{2\alpha}} Z_{[nt]})_{t\ge 0}\mathop{\Longrightarrow}_{n\rightarrow\infty}
^{\mathcal{L}} \left(\Delta_t : =  \int_{\mathbb R} L_t(x) \, \dd W(x)\right)_{t\ge 0},$$
where $(L_t(x))_{x\in\mathbb{R},t\geq 0}$ is a continuous version with compact support (for every $t$) of the local time of the process $Y$ (see \cite{marcusrosen}).

Hence in any of the cases considered above, $((Z_{\lfloor nt\rfloor}/a_n)_{t\ge 0})_ {n\ge 0}$ converges in distribution (with respect to the
$J_1$-metric) to some process $\Delta$, with
\begin{equation} \label{eqn:a_nforrwrs}
a_n:=\left\{ \begin{array}{lll}
n^{1- \frac{1}{2\alpha}} & \text{if} & S\text{ is recurrent and }
 d=1,\ \alpha\in (1,2].\\
\sqrt{n\log n} & \text{if }& S\text{ is recurrent and }
\alpha=d\in\{1,2\}.\\
\sqrt{n} & \text{if} & S\text{ is transient}.
\end{array}
\right.
\end{equation}
For every $y\in\mathbb Z^d$ and every integer $n\ge 1$, we write $N_{n}(y)$ for the number of visits
of the walk $S$ to site $y$ before time $n$, i.e.
$$ N_n(y):=\#\{k=1,\ldots,n\ :\ S_k=y\}.$$
We also write $R_n:=\#\{S_1,\ldots,S_n\}$ for the range of $S$ up to time $n$.
Note that $Z$ can be rewritten as follows:
$$Z_n=\sum_{y\in\mathbb Z^d}\xi_yN_n(y).$$
In this context, we prove the following result
valid under our general assumptions on the walk $S$
(in any dimension $d$ and with $a_n$ given by \eqref{eqn:a_nforrwrs}).
\begin{thm}[Persistence probability for RWRS, first result]\label{rwrs}
\hspace{10cm}\begin{itemize}
\item If $\xi$ is Gaussian, then
 \begin{equation}\label{Gauss}
 \forall m\in\mathbb R,\quad \exists c>0,\ \ \forall n\ge 2,\ \ 
 c^{-1} a_n/(n\sqrt{\log n})\le \mathbb P\left(Z_n^*\le m\right)\le c\, a_n/n.
 \end{equation}
\item Let us assume that $\xi$ is bounded from below. Then,
\begin{itemize}
\item for all $m\in \mathbb R$, there exists a $c>0$ such that
\begin{equation}\label{bounded1}
\forall n\ge 1,\ \mathbb P\left( Z_n^*\le m\right)\le c\, a_n/n.
\end{equation}
\item for all $m\geq 0$, there exists a $c>0$ such that
\begin{equation}\label{bounded2}
\forall n\ge 1,\ \mathbb P\left( Z_n^*\le m\right)\ge c\, a_n/n.
\end{equation}
\item for all $m< 0$ such that $\mathbb P(\xi_1\leq m)>0$, there exists a $c>0$ such that
\begin{equation}\label{bounded3}
\forall n\ge 1,\ \mathbb P\left( Z_n^*\le m\right)\ge c\, a_n/n.
\end{equation}\end{itemize}
\end{itemize}
\end{thm}
We first prove the following results in oder to apply our general theorems.
\begin{prop}\label{LEM0}
As $n$ tends to infinity,
$$\forall\beta\in(1,2),\ \mathbb E\left[
\max_{j=1,...,n}|Z_j|
^{\beta}\right]=O(a_n^{\beta})\quad\mbox{and}\quad\lim_{n\rightarrow +\infty}\frac{\mathbb E\left[Z_n^*\right]}{a_n}=\mathbb E\left[\sup_{t\in[0,1]}\Delta_t \right]\, .$$
\end{prop}
\begin{proof}
Due to the convergence for the $J_1$-topology of $((a_n^{-1} Z_{\lfloor nt\rfloor})_t)_{n\ge 1}$ to $(\Delta_t)_t$ as $n$ goes to infinity,
we know that $(a_n^{-1}Z_n^*)_{n\ge 1}$ converges in distribution
to $\sup_{t\in[0,1]}\Delta_t$
as $n$ goes to infinity (see Section 12.3 in \cite{Whitt}).
Let us prove that $(a_n^{-1}Z_n^*)_{n\ge 1}$ is
uniformly integrable. 
To this end we will use the fact that, conditionally to
the walk $S$, the increments of $(Z_n)_ {n\ge 0}$ are centered and positively associated.
Due to Theorem 2.1 of \cite{Gong} (applied with $\Phi=id$, $g(x)=|x|^{\beta}$, $c_k= 1$, $p=2/\beta$), there exists some constant 
$\tilde c>0$ such that
\[
 \mathbb E\left[\max_{j=1,\ldots,n} |Z_j|^{\beta} |S\right]   
 \le \tilde c \, \mathbb E\left[|Z_n|^2|S\right]^{\frac \beta 2},
\]
so, using the H\"older inequality, $\mathbb E\left[
\max_{j=1,\ldots,n} |Z_j|^{\beta}\right] \le \tilde c\,  \mathbb E\left[|Z_n|^2\right]^{\frac \beta 2}$.
It remains to prove that $\mathbb E[|Z_n|^2]=O(a_n^2)$. We observe that 
$\mathbb E\left[|Z_n|^{2}\right]=\mathbb E[\xi_0^2]\mathbb E[V_n]$, 
where $V_n$ is the number of self-intersections up to time $n$
of the random walk $S$, i.e.\ 
$V_n=\sum_x(N_n(x))^2=\sum_{i,j=1}^n{\mathbf 1}_{S_i=S_j}$. 
Standard computations (see Lemma 2.3 in \cite{Bo89} for the case $d=2$, which are exactly the same in $d=1$) give that
$$\mathbb E[V_n]=\sum_{i,j=1}^n\mathbb P(S_{|i-j|}=0) 
   =n+2\sum_{k=1}^n(n-k)\mathbb P(S_k=0) \sim c'(a_n)^2$$
(since $\mathbb P(S_k=0)\sim c_0 k^{-\frac d\alpha}$ when $S$
is recurrent and since $\sum_k\mathbb P(S_k=0)<\infty$ when $S$ is transient \cite{Spitzer})
and the result follows.
\end{proof}
\begin{proof}[Proof of Theorem \ref{rwrs}]
Let us prove $(\ref{Gauss})$. Since the increments of $Z$ are positively associated and identically distributed conditionally to the random walk $S$, 
due to Lemma~\ref{lem:positiveassociation}, it is enough to prove the lower bound for the level $0$ and the upper bound for the negative levels $m$ (Remark that here $\mathbb P(\xi_1\leq m)>0$ for any $m\in \mathbb R$).\\
Proposition \ref{LEM0} ensures that $\mathbb E[Z_n^*]\sim \tilde c_0\, a_n$ for some $\tilde c_0>0$.
Therefore Theorem \ref{THMlem} ensures that, for every $m>0$, $\mathbb P(Z_n^*\le -m)=O(a_n/n)$.\\
Moreover, due to Proposition \ref{LEM0},
$$\Vert Z_n^*-Z_n\Vert_{\beta}=\mathcal O(a_n),$$
so the lower bound for the level $m=0$ comes from
Theorem \ref{thm:lower}-(ii) applied with $p=\beta$, together with Remark \ref{MAINTHM}.\\
Inequalities (\ref{bounded1}), (\ref{bounded2}) and (\ref{bounded3}) can be proved in the same way by remarking that our assumptions on the scenery implies that there exists some $m'>0$ such that $\mathbb P(\xi_1<-m')>0.$
\end{proof}
In the case of RWRS, the proof of Theorem  \ref{thm:lower} can be modified
in order to get better lower bounds. 
\begin{prop}[Persistence probability for RWRS, general case]\label{lowerRWRS}
Let $\beta' >1$ be such that $\mathbb E[\max(0,- \xi_0)^{\beta'}]<\infty$. Then,
\begin{itemize}
\item for all $m\in \mathbb R$, there exists a $c>0$ such that
\begin{equation}\label{bounded1bis}
\forall n\ge 1,\ \mathbb P\left( Z_n^*\le m\right)\le c\, a_n/n.
\end{equation}
\item for all $m\geq 0$, there exists a $c>0$ such that
\begin{equation}\label{bounded2bis}
\forall n\ge 1,\ \mathbb P\left( Z_n^*\le m\right)\ge c\, \left(               \frac {a_n}n\right)^
{
1+\frac{2}{\beta'}
}.
\end{equation}
\item for all $m< 0$ such that $\mathbb P(\xi_1\leq m)>0$, there exists a $c>0$ such that
\begin{equation}\label{bounded3bis}
\forall n\ge 1,\ \mathbb P\left( Z_n^*\le m\right)\ge c\, \left(               \frac {a_n}n\right)^
{
1+\frac{2}{\beta'}
}.
\end{equation}\end{itemize}
\end{prop}
\begin{rqe}
If $\xi_0$ admits moments of every order then 
\[
(a_n/n)^{1+o(1)}\le \mathbb P\left(Z_n^*\le m\right)=O(a_n/n) ,
\quad as\ n\rightarrow +\infty .\]
Indeed writing $\mathbb P(Z_n^*\le m)=(a_n/n)^{\vartheta_n}$,
we get that for every $\beta'>1$, 
$$\limsup_{n\rightarrow +\infty}\vartheta_n\le 1+\frac{2}{\beta'}.$$
\\
By taking $\beta'=2$, the conclusion of Proposition \ref{lowerRWRS} becomes
$$\exists C>1,\quad \forall n\ge 1, \quad C^{-1}\left(\frac{a_n}n\right)^2
        \le \mathbb P(Z_n^*\le m)\le C\frac{a_n}n.$$
\end{rqe}
\begin{proof}[Proof of Proposition \ref{lowerRWRS}]
We just have to prove the lower bound.
It is enough to prove it for $m=0$ (due to Lemma~\ref{lem:positiveassociation}). 
Let $\beta\in(1,2)$. From H\"{o}lder's inequality, we have
\begin{eqnarray*}
&\ &\mathbb E\left[(Z_{n+\lfloor \varepsilon n\rfloor}^*-Z_{n+\lfloor \varepsilon n\rfloor}) 
\mathbf   1_{\{\sup_{k=1,\ldots,\lfloor
         \varepsilon n\rfloor}   (Z_{k}-Z_{k+1})> b_{dn}\}}\right]\\
&\le &\Vert Z_{n+\lfloor \varepsilon n\rfloor}^*-Z_{n+\lfloor \varepsilon n\rfloor}\Vert_{\beta}\left(     \mathbb P\left(\sup_{k=1,\ldots,\lfloor        \varepsilon n\rfloor}   (Z_{k}-Z_{k+1})> b_{dn}\right)\right)^{1-\frac 1\beta}\, .
\end{eqnarray*}
But, setting $\mathcal F_0$ for the $\sigma$-algebra generated by $S$,
\begin{eqnarray*}
\mathbb P\Big(\sup_{k=1,\ldots,{\lfloor \varepsilon n\rfloor}}(Z_{k}-Z_{k+1})> b_{dn} \Big|  \mathcal F_0\Big)&\le &
R_{\lfloor \varepsilon n\rfloor}\  \mathbb P(-Z_1>b_{dn})\\
&\le& R_{\lfloor \varepsilon n\rfloor}\ 
\mathbb E[\max(0,-\xi_0)^{\beta'}] \,  b_{dn}^{-\beta'}.
\end{eqnarray*}
Therefore 
\begin{eqnarray*}
&\ &\mathbb E\left[(Z_{n+\lfloor \varepsilon n\rfloor}^*-Z_{n+\lfloor \varepsilon n\rfloor}) 
\mathbf   1_{\{\sup_{k=1,\ldots,\lfloor
         \varepsilon n\rfloor}   (Z_{k}-Z_{k+1})> b_{dn}\}}\right]\\
&\le &\Vert Z_{n+\lfloor \varepsilon n\rfloor}^*-Z_{n+\lfloor \varepsilon n\rfloor}\Vert_{\beta}\left( \mathbb E[R_{\lfloor \varepsilon n\rfloor}] \mathbb E[\max(0,-\xi_0)^{\beta'}] \,  b_{dn}^{-\beta'}\right)^{1-\frac 1\beta}\, .
\end{eqnarray*}
Due to Proposition \ref{LEM0},
$
 \left\Vert Z_{m}^*-Z_{m}\right\Vert_{\beta}=  O(a_m).
$
Now we choose $b_n=\mathbb E[R_{\lfloor \varepsilon n\rfloor}]^{\frac 1{\beta'}}$ and $d$ large enough such that
$$\limsup_{n\rightarrow +\infty} (\mathbb E[Z_n^*])^{-1}\left\Vert Z_{n+\lfloor \varepsilon n\rfloor}^*-Z_{n+\lfloor \varepsilon n\rfloor}\right\Vert_{\beta}
      \left(\mathbb E[ R_{\lfloor \varepsilon n\rfloor} ]
         \mathbb E\Big[\max(0,-\xi_0)^{\beta'}\Big]
       b_{dn}^{-\beta'}\right)^{1-\frac 1\beta}<\lim_{n\rightarrow +\infty}\frac{\mathbb E[Z_{n+\lfloor \varepsilon n\rfloor}^*-Z^*_{n}]}{\mathbb E[Z_n^*]},$$
and so
\eqref{Hypcompliquee} is satisfied, which implies, due to Remark
\ref{Condcompl}, that $\liminf_{n\rightarrow +\infty}
      \frac {n b_{dn}}{\mathbb E[Z_n^*]}  \mathbb P\Big(Z_n^*< 0\Big)>0$.
But $\mathbb{E}[R_n]\sim c (\frac{n}{a_n})^{2}$.
Indeed there exists a $c_1>0$
such that $\mathbb E[R_n]\sim c_1 n$ if $S$ is transient
(\cite[p. 36]{Spitzer}),
$\mathbb E[R_n]\sim c_1 n/\log n$ if $S$ is recurrent and $\alpha=d$ and
$\mathbb E[R_n]\sim c_1 n^{\frac 1\alpha}$ if 
$S$ is recurrent and $\alpha>d=1$
(see \cite[p. 698, 703]{LGR}).
 Hence $b_n=O( (\frac{n}{a_n})^{\frac{2}{\beta'}} )$ and so $\liminf_{n\rightarrow +\infty}\left(               \frac n{a_n}\right)^{1+\frac{2}{\beta'}}
\mathbb P\Big(Z_n^* \leq m \Big)>0$.
Since $\mathbb P\Big(Z_n^* \leq m \Big)$ is decreasing in $n$, we conclude that $\mathbb P\Big(Z_n^* \leq m \Big)>0$ for every $n$ and we obtain the lower bound for every $n$.
\end{proof}

We are now interested in the case when the random variables $\xi$ are $\mathbb Z$-valued. Better estimates can be obtained 
for $\mathbb P(T_0>n)$
in this context.
When $Z_1$ takes its values in $\mathbb Z$, we define 
the range ${\mathcal R}_n$ of the RWRS $Z$, i.e.
the number of sites visited by $Z$ before up to time $n$, by 
$${\mathcal R}_n:=\#\{ Z_1,\dots,  Z_{n}\}.$$
\begin{prop}\label{cvpsrange4}
Assume that $\mathbb P(\xi_1\in\mathbb Z)=1$. If $S$ is recurrent, then 
\begin{equation}\label{EEE2a}
0<\liminf_{n\rightarrow +\infty}\frac{\mathbb E[\mathcal R_n]}{a_n}\leq \limsup_{n\rightarrow +\infty}\frac{\mathbb E[\mathcal R_n]}{a_n}<\infty\, ,
\end{equation}
\begin{equation}\label{EEE3a}
0<\liminf_{n\rightarrow +\infty}\frac n{a_n}\mathbb P(T_0>n)\leq \limsup_{n\rightarrow +\infty}\frac n{a_n}\mathbb P(T_0>n)<\infty\, .
\end{equation}
\end{prop}
\begin{proof}
Since $\mathcal R_n\le\max_{k=1,\ldots,n}Z_k-\min_{k=1,\ldots,n}Z_k+1$, we already know from Proposition \ref{LEM0} that $\limsup_{n\rightarrow +\infty}\mathbb E[\mathcal R_n]/a_n<\infty$. 
Let us prove that $\liminf_{n\rightarrow +\infty}
\mathbb E[\mathcal R_n]/a_n>0$.
Let $\mathcal N_n(x):=\#\{k=1,\ldots,n\,:\,Z_k=x\}$.
Applying the Cauchy-Schwarz inequality to 
$n= \sum_x  \mathcal N_n(x) \mathbf 1_{\{\mathcal N_n(x)>0\}} $, we obtain
$$n^2\le \sum_y \mathbf 1_{\{\mathcal N_n(y)>0\}}\, \sum_x(\mathcal N_n(x))^2 =\mathcal R_n\, \mathcal V_n,$$
with $\mathcal V_n=\sum_{x}(\mathcal N_n(x))^2
         =\sum_{i,j=1}^n\mathbf 1_{\{Z_i=Z_j\}}$
the number of self-intersections of $Z$ up to time $n$
and so using Jensen's inequality,
$\frac{ \mathbb E[\mathcal R_n]}{a_n} \ge \frac{n^2}{a_n}  \mathbb E[(\mathcal V_n)^{-1}]\ge \frac{n^2}{a_n}{\mathbb E[\mathcal V_n]^{-1}}$.
Moreover, using the local limit theorems for the RWRS proved in \cite[Theorem 1]{TLL} and in \cite[Theorem 3]{FFN},
\begin{equation}\label{EVn}
\mathbb E[\mathcal V_n]=
    n+ 2 \sum_{1\leq i <j \leq n} \mathbb P(Z_{j-i}=0)  \sim    C' \frac{n^2}{a_n},
\end{equation}
so $\liminf_{n\rightarrow +\infty} \frac{\mathbb E[\mathcal R_n]}{a_n }\ge \frac 1{C'} >0$. This gives \eqref{EEE2a}. 
Finally \eqref{EEE3a} follows from Theorem \ref{THMZ}.

\end{proof}

In the particular case where $\mathbb P(\xi_0  \in\{-1,0,1\})=1$, we obtain a precise estimate (as a consequence of the two last items of Theorem \ref{THMZ}).
\begin{thm}\label{cvpsrange3}
Assume that $\mathbb P(\xi_1\in\{-1,0,1\})=1$,
then
\begin{equation}\label{EEE1}
\frac{\mathcal R_n}{a_n} \stackrel{\mathcal L}{\longrightarrow} \sup_{t\in[0,1]}\Delta_t-\inf_{t\in[0,1]}\Delta_t\quad\mbox{and}\quad
\lim_{n\rightarrow+\infty} \frac{\EE[\mathcal R_n]}{a_n}=2\, \mathbb E\left[\sup_{t\in[0,1]}\Delta_t\right]   \, ,
\end{equation}
\begin{equation}\label{EEE3}
\mathbb P(Z_n^*\le -1)
=\frac 12 \mathbb P(T_0>n)\sim \frac {\gamma a_n} {n}\mathbb E\left[\sup_{t\in[0,1]}\Delta_t\right],
\end{equation}
where $\gamma=1-\frac{1}{2\alpha}$ when $d=1<\alpha$ and $\gamma=1/2$ otherwise. \end{thm}
\begin{proof}
Since $\mathbb P(\xi_1\in\{-1,0,1\})=1$, 
$\mathcal R_n=Z_n^*-(-Z)_n^*+1$.
We deduce from Proposition \ref{LEM0} the convergence of $(\mathbb E[\mathcal R_n]/a_n)_ {n\ge 0}$. The convergence in distribution
of $(\mathcal R_n/a_n)_ {n\ge 0}$ comes from the convergence in distribution of $(Z_{\lfloor n\cdot\rfloor}/a_n)_n$ for the $J_1$-metric.
The last part of Theorem \ref{cvpsrange3} follows from \eqref{EEE1} and from the last item of Theorem \ref{THMZ}
since $(a_n)_ {n\ge 0}$ is regularly varying with exponent $\gamma$ where $\gamma=1-\frac{1}{2\alpha}$ when $S$ is recurrent and 
$d=1<\alpha$ and $\gamma=1/2$ otherwise, and since $Z$ has the same distribution as $-Z$. 
\end{proof}

\begin{rqe}
It is worth noticing that the techniques we used in this section can be applied to more general RWRS, for instance to RWRS
studied in \cite{Wang}. 
\end{rqe}
\section{Random process in Brownian scenery}\label{RPBS}
Let us consider generalizations of the Kesten-Spitzer process $(\Delta_t)_{t\geq 0}$ introduced in the previous section. 
Let $W=\acc{W(x); x\in\mathbb{R}}$ 
be a two-sided real Brownian motion
and $Y=\acc{Y(t); t\ge 0}$ be a real-valued self-similar process of index $\gamma \in (0, 2)$ with stationary increments. 
 We assume that there exists a continuous version $\acc{L_t(x); x\in\mathbb{R},t\geq 0}$ of the local time of $Y$. 
The processes $W$ and $Y$ are defined on the same probability space and are assumed to be independent.
We consider {\it the random process in Brownian scenery} $\acc{\Delta_t;t\geq 0}$ defined as
\[\Delta_t = \int_{\mathbb{R}} L_t(x) \, dW(x).\]
The process $\Delta$ is itself a self-similar process of index $h$ with stationary increments, with
\[h:=  1-\frac{\gamma}{2}.\] 
Let $V_1:=\int L_1^2(x) \, dx$ be the self-intersection local time of $Y$.
Given $Y$, $\Delta_1$ has centered Gaussian distribution
with variance $V_1$.
The following assumption is made on the random variable $V_1$.
\begin{description}
\item[(H1)]  There exist a real number $\alpha > 1$, and positive constants $C,c$ such that for any $x \ge 0$,
$$\PP\Big(V_1 \ge x\Big) \le C \exp(-cx^{\alpha}) .$$
\end{description}
In \cite{castellguillotinwatbled}, examples of processes $Y$ satisfying the above assumptions are given: stable L\'evy process with index $\delta \in (1,2]$ (it satisfies (H1) with $\alpha=\delta$, see Lemma 2.2 in \cite{castellguillotinwatbled}), 
the fractional Brownian motion with index $H\in (0,1)$ (it satisfies (H1) with $\alpha=\frac{1}{H}$, see Lemma 2.3 in \cite{castellguillotinwatbled}) and the iterated Brownian motion which satisfies assumption (H1) with $\alpha=\frac{4}{3}$ (see
 Lemma 2.4 in \cite{castellguillotinwatbled}).

Our main result of this section is the following one.
\begin{thm}\label{theoMk2}
Let $a>0$. Assume (H1) and that there exists a $c_a>0$
such that, for every $n\geq 1$,
\begin{equation}\label{Quasi-ass}
\mathbb P\left(\Delta_1\le -2a,\ \max_{k=1,...,n}(\Delta_{k+1}-\Delta_1)\le a\right)
\ge c_a \mathbb P\left(\max_{k=1,...,n}(\Delta_{k+1}-\Delta_1)\le a\right)\, .
\end{equation}
Then, there exists a constant $c>0$,  such that for large enough $T$ , 
\begin{equation}\label{eqMk}
 c^{-1}\ T^{-\frac{\gamma}{2}} (\ln T)^{-\frac{1}{2-\gamma}(1+\frac{1}{\alpha}) }\leq \PP\Big( \sup_{t\in[0,T]} \Delta_t\leq a\Big)\leq c T^{-\frac{\gamma}{2}}.
\end{equation}
\end{thm}
Note that $\PP\Big( \sup_{t\in[0,T]} \Delta_t< 0\Big)=0$.

Let us notice that $\Delta$ satisfies \eqref{Quasi-ass}
in the particular case where $Y$ is a stable L\'evy process (see Proposition 3.1 in \cite{KL}).
Before proving Theorem \ref{theoMk2}, let us state a preliminary result.
Up to enlarging the probability space, we consider
a process $\widetilde Y$, independent of $W$, such that, for every $T>0$,
$(\widetilde Y_t)_{t\in[0,T]}$ has the same distribution as
$(Y_{T-t}-Y_{T})_{t\in[0,T]}$ (this is possible, due to the Kolmogorov theorem, since, for every $0<T<T'$, $(Y_{T-t}-Y_{T})_{t\in[0,T]}$ and $(Y_{T'-t}-Y_{T'}=Y_{T-t+(T'-T)}-Y_{T+(T'-T)})_{t\in[0,T]}$ have the same distribution since the increments of $Y$ are stationary).
\begin{lem}
For every $T>0$, 
The process $(\Delta_{T-t}-\Delta_{T})_{t\in[0,T]}$ has the same distribution as $\left(\widetilde\Delta_t:=\int_{\mathbb R}\widetilde{L}_t(x)\, dW(x)\right)_{t\in[0,T]}$ with $\tilde L$ the local time of
$\tilde Y$.\\
Moreover $\widetilde V_1:=\int_{\mathbb R}\widetilde L^2_1(x)\, dx$ has the same distribution as $V_1$.
\end{lem}
\begin{proof}Observe first that
\begin{eqnarray*}
\int_{[0,t]}f(Y_{T-s}-Y_{T})\, ds &=& \int_{[T-t,T]}f(Y_u-Y_T)\, du\\
 &=& \int_{[0,T]}f(Y_u-Y_T)\, du-\int_{[0,T-t]}f(Y_u-Y_T)\, du\\
&=&\int_{\mathbb R} f(x-Y_T)\, L_T(x)\, dx- \int_{\mathbb R} f(x-Y_T)\, L_{T-t}(x)\, dx\\
&=&\int_{\mathbb R} f(y)\, L_T(Y_T+y)\, dy- \int_{\mathbb R} f(y)\, L_{T-t}(Y_T+y)\, dy\\
&=&\int_{\mathbb R} f(y)\, \left(L_T(Y_T+y)-L_{T-t}(Y_T+y)\right)\, dy\, .
\end{eqnarray*}
Therefore the local time $L^{(T)}$ of $(Y_{T-t}-Y_{T})_{t\in[0,T]}$ is given by $L^{(T)}_t(x)=L_T(Y_T+x)-L_{T-t}(Y_T+x)$ and so
\begin{eqnarray*}
\Delta_{T-t}-\Delta_{T}&=&\int_{\mathbb R}(L_{T-t}(x)-L_{T}(x))\, dW(x)\\
&=&-\int_{\mathbb R}L^{(T)}_t(x-Y_T)\, dW(x)\\
&=&\int_{\mathbb R}L^{(T)}_t(y)\, dW'(y)\, ,
\end{eqnarray*}
with $W'(y):=W(Y_T)-W(Y_T+y)$.
But $(L^{(T)}_t(x),W'(x))_{t\in[0,T],\ x\in\mathbb R}$ has the same distribution as $(\widetilde L_t(x), W(x))_{t\in[0,T],\ x\in\mathbb R}$.
This ends the proof of the first point.\\
For the second point, we observe that $\widetilde V_1:=\int_{\mathbb R}\widetilde L^2_1(x)\, dx$ has the same distribution as
\begin{eqnarray*}
\int_{\mathbb R}(L^{(1)}_1(x))^2\, dx
     &=& \mathbb E\left[\int_{\mathbb R}(L_1(Y_1+x)-L_0(Y_1+x))^2\, dx\right]\\
&=& \int_{\mathbb R}(L_1(y))^2\, dy =V_1\, .
\end{eqnarray*}
\end{proof}
\begin{proof}[Proof of Theorem \ref{theoMk2}]
Up to multiplying $W$ by a positive constant, we assume that $a=1$. 
Due to \eqref{Quasi-ass},
\begin{eqnarray}\label{qa}
\mathbb P\left(\max_{k=1,...,n+1}\Delta_k\le -1\right)
&\ge&\mathbb P\left(\Delta_1\le -2,\ \max_{k=1,...,n}(\Delta_{k+1}-\Delta_1)\le 1\right)\nonumber\\
&\ge& c_a\, \mathbb P\left(\Delta_{n}^*\le 1\right)\, .
\end{eqnarray}
Observe first that inequalities (3.6) and (3.2) in \cite{castellguillotinwatbled} ensure the existence of two positive constants $\tilde a$ and $\tilde C$ such that
\begin{equation}\label{TAILDECAY1}
\forall x>0,\quad
\mathbb P\left(\sup_{[0,1]} \Delta>x\right)\le 2\, 
\mathbb P\left(\Delta_1 >x\right) \le 2\, \tilde C \, e^{-\tilde a x^{\frac{2 \alpha}{1+\alpha}}  }\, 
\end{equation}
and
\begin{equation}\label{TAILDECAY2}
\forall x>0,\quad
\mathbb P\left(\sup_{[0,1]} (-\Delta)>x\right)\le 2\, 
\mathbb P\left(-\Delta_1 >x\right) \le 2\, \tilde C \, e^{-\tilde a x^{\frac{2 \alpha}{1+\alpha}}  }\, .
\end{equation}
Therefore, since $\sup_{[0,1]}\Delta\ge 0$, we have
\begin{eqnarray}
\mathbb E[\sup_{[0,1]}| \Delta| ]&=&\int_{(0,+\infty)}\mathbb P
   \left(\sup_{[0,1]} |\Delta|>x\right)\, dx\nonumber\\
&\le&\int_{(0,+\infty)}\left(\mathbb P
   \left(\sup_{[0,1]} \Delta>x\right)+\mathbb P
   \left(\sup_{[0,1]} (-\Delta)>x\right)\right)\, dx\nonumber\\
   &\le& 4\, \tilde C \int_{(0,+\infty)}\exp\left( -\tilde a\, x^{\frac{2 \alpha}{1+\alpha}}\right)\, dx<\infty\, .\label{chgtlevel}
\end{eqnarray}
Hence, since $\Delta$ is almost surely continuous (and therefore uniformly continuous) on $[0,1]$, $\max_{k=1,...,n}\Delta_{\frac k n}$ converges almost surely to $\sup_{[0,1]}\Delta$, and so, with the use of the dominated convergence theorem, we conclude that
\begin{equation}\label{equivEspSup}
\mathbb E[\max_{k=1,...,n}\Delta_k]=n^h\mathbb E[\max_{k=1,...,n}\Delta_{\frac k n}]\sim n^h\mathbb E[\sup_{[0,1]}\Delta]\, .
\end{equation}
Therefore, the discrete-time process
$(\Delta_n)_{n\ge 0}$ (with $\Delta_0=0$) satisfies the assumptions of Theorem \ref{THMlem} and so, due to \eqref{qa},
we deduce the upper bound
$$ \PP\Big( \sup_{t\in[0,T]} \Delta_t\leq 1\Big)\leq \PP\Big( \max_{k=1,\ldots,\lfloor T\rfloor} \Delta_k  \leq 1\Big)
\le \frac{\mathbb P(\max_{k=1,...,\lfloor T\rfloor +1}\Delta_k\le -1)}{c_a}
=\onotation( T^{-\frac{\gamma}{2}} ).$$
Moreover, \eqref{equivEspSup} ensures that \eqref{eqn:notregvar} is also satisfied and the last inequality in \eqref{TAILDECAY2} ensures that
\eqref{eqn:hypbn} is satisfied with $b_n= c_0\ (\log n)^{(1+\alpha)/2\alpha}$ for a suitable $c_0$.
By reasoning as above, we deduce that
$$
\mathbb E[\max_{k=1,...,n}|\Delta_k\vert^2]
= n^{2h}   \mathbb E[\max_{k=1,...,n}|\Delta_{\frac k n}|^2]\\
\sim  n^{2h}  \mathbb E[\sup_{[0,1]}|\Delta|^2]\, .
$$
Therefore, $\vert\vert \Delta_n^{*}-\Delta_n\vert\vert_2=\mathcal O(n^h)$
and $(\Delta_n)_{n\ge 0}$ satisfies the assumptions of Item (ii) of Theorem \ref{thm:lower} with $b_n= c_0\ (\log n)^{(1+\alpha)/2\alpha}$ for a suitable $c_0$.
Then,
\begin{equation}\label{lowerDelta}
\liminf_{n\rightarrow +\infty} n^{\frac{\gamma}{2}} (\log n)^{\frac{1+\alpha}{2\alpha}} \PP\Big( \max_{k=1,\ldots,n} \Delta_k  < 0\Big) >0.
\end{equation}
Let us prove the lower bound in Theorem \ref{theoMk2}. Proceeding as in the proof of Theorem \ref{fBM} 
we set $M_n:=\lfloor(C \log n)^{\frac{1+\alpha}{2\alpha h}}\rfloor$ for some $C>0$, $n_k:=k/M_n$ and 
$A_k:=\max_{s\in [n_k,n_{k+1}]}\left( \Delta_s - \Delta_{n_{k+1}}\right)$. Observe that
$$\max_{t\in[0,n]}\Delta_t-
\max_{k=1,\ldots,n M_n}\Delta_{k/M_n}
      \le \max_{k=0,\ldots,n M_n-1}A_k,$$
Then, using the stationarity and the self-similarity of $\Delta$, 
\begin{eqnarray*}
&\ &\mathbb P\left(\max_{k=0,\ldots,n M_n-1}A_k>a\right)\le n\, M_n\mathbb P\left(\max_{t\in [0,1]}(\Delta_{tn_1}-\Delta_{n_1})>a \right)\\
&\le& n(C \log n)^{\frac{1+\alpha}{2\alpha h}}\mathbb P\left(\max_{t\in [0,1]}(\Delta_{t}-\Delta_{1})>a(M_n)^h\right)\\
&\le& n(C \log n)^{\frac{1+\alpha}{2\alpha h}}\mathbb P\left(\max_{[0,1]}\widetilde\Delta >a(M_n)^{h}\right)\\
&\le& n(C \log n)^{\frac{1+\alpha}{2\alpha h}}2\, \mathbb P\left(\widetilde\Delta_1 >a(M_n)^{h}\right)\\
&\le& O\left(n\, (\log n)^{\frac{1+\alpha}{2\alpha h}} 
e^{-\tilde a_0 M_n^{\frac{2 \alpha h}{1+\alpha}}  }\right)\\
&\le& O\left(n\, (\log n)^{\frac{1+\alpha}{2\alpha h}} 
e^{-\tilde a_0 C\log n}  \right)\, ,
\end{eqnarray*}
for some $\tilde a_0>0$ (applying (3.6) and (3.2) of \cite{castellguillotinwatbled} to the process $\widetilde\Delta$).
We choose $C=C(a)$ large enough so that this last quantity is in $o(n^{-\frac\gamma 2}(\log n)^{-\frac {1+\alpha}{2\alpha h}})$. So,
using the self-similarity of $\Delta$, we obtain
\begin{eqnarray*}
&& \mathbb P\left(\max_{t\in[0,n]}\Delta_t\le a\right)\\
&\ge& \mathbb P\left(\max_{k=1,\ldots,nM_n}\Delta_{k/M_n}\le  0\right)-o(n^{-\frac\gamma 2}(\log n)^{-\frac {1+\alpha}{2\alpha h}})\\
&\ge& \mathbb P\left((M_n)^{-h}
        \max_{k=1,\ldots,nM_n}\Delta_k\le 0\right)-o(n^{-\frac\gamma 2}(\log n)^{-\frac {1+\alpha}{2\alpha h}})\\
&\ge&\mathbb P\left(\max_{k=1,\ldots,nM_n}\Delta_k\le 0\right)-o(n^{-\frac\gamma 2}(\log n)^{-\frac {1+\alpha}{2\alpha h}})\\
&\ge& {c'}(nM_n)^{-\frac \gamma 2}(\log n)^{-\frac{1+\alpha}{2\alpha}}-o(n^{-\frac\gamma 2}(\log n)^{-\frac {1+\alpha}{2\alpha h}})\, ,
\end{eqnarray*}
for some $c'>0$ due to \eqref{lowerDelta}.
We conclude by noticing that
$$(nM_n)^{-\frac \gamma 2}(\log n)^{-\frac{1+\alpha}{2\alpha}}
\sim C^{-\frac{(1+\alpha)\gamma}{4\alpha h}}
n^{-\frac\gamma 2}(\log n)^{-\frac{1+\alpha}{2\alpha}\left(1+\frac{\gamma}{2h}\right)}\, .$$
This ends the proof of the lower bound of \eqref{eqMk} since $2h=2-\gamma$ and so $\frac{1+\alpha}{2\alpha}(1+\frac\gamma{2h})=\frac 12\left(1+\frac 1\alpha\right)\frac 2{2-\gamma}=\frac 1{2-\gamma}\left(1+\frac 1\alpha\right)$.
\end{proof}
The following result improves \cite{BFFN,castellguillotinwatbled}.
\begin{coro}[Persistence probability for Random process in Brownian scenery]\label{thm:RPBS}
If $Y$ is a stable L\'evy process with index $\alpha\in(1,2]$, then there exists a $c>0$ such that
\[
c^{-1}\, T^{-\frac 1{2\alpha}}(\log T)^{-\frac{(1+\alpha)}{(2\alpha-1)}}\le \mathbb P\left(\sup_{t\in[0,T]}
         \Delta_t\le 1\right)=O(T^{-\frac 1{2\alpha}}) .
\]
\end{coro}

\section{The Matheron-de Marsily model} \label{sec:mdm}
Finally, we will consider particular models of random walks $(M_n)_ {n\ge 0}$ in random environment on $\mathbb Z^2$.
We are namely interested in the survival probability of a particle evolving on a randomly oriented lattice 
introduced by Matheron and de Marsily in \cite{MdM} (see also \cite{BGKPS}) to model fluid transport in a porous stratified medium. 
Supported by physical arguments, numerical simulations and comparison with the fractional Brownian motion, 
Redner \cite{Red2} and Majumdar \cite{Maj} conjectured that the survival probability asymptotically behaves as $n^{-\frac{1}{4}}$. 
In this paper we rigorously prove their conjecture.  

Let us describe more precisely the model and the results. Let us fix $p\in(0,1)$.
The (random) environment will be given by a sequence $\xi=(\xi_k)_{k\in\mathbb Z}$
of i.i.d.\ centered random variables with
values in $\{\pm 1\}$ and defined on the probability space
$(\Omega,\mathcal T,\p)$. Given $\xi$, the position $M$ of the particle is defined as a
$\mathbb Z^2$-random walk on nearest neighbours starting
from $0$ (i.e.\ $\mathbb P^\xi(M_0=0)=1$) and with transition probabilities
$$\mathbb P^\xi(M_{n+1}=(x+ \xi_y,y)|M_n=(x,y))=p,\quad
   \mathbb P^\xi(M_{n+1}=(x,y\pm 1)|M_n=(x,y))=\frac{1-p}2. $$
At site $(x,y)$, the particle can either get down  (or get up) with probability $\frac{1-p}2$ or move with probability $p$ on the $y$-th horizontal line according to its orientation (to the right (resp.\ to the left) if $\xi_y=+1$ (resp.\ if $\xi_y=-1$)).
We will write $\mathbb P$ for the annealed law, that is the
integration of the quenched distribution $\mathbb P^\xi$ with respect to $\p$.

In the sequel, this random walk will be named MdM random walk.
This 2-dimensional random walk in random environment was first studied rigorously in \cite{CP}. They proved that the MdM random walk is transient under the annealed law $\mathbb P$ and under the quenched law $\mathbb P^\xi$ for $\p$-almost every environment $\xi$. It was also proved that it has speed zero.
Actually, the MdM random walk is closely related to RWRS. This fact was first noticed in \cite{GPN}. More precisely its first coordinate can be viewed as a generalized RWRS; the second coordinate being a "lazy random walk" on $\mathbb Z$ (see Section 5 of \cite{TLL} for the details). Using this remark, a functional limit theorem was proved in \cite{GPN} and a local limit theorem was established in \cite{TLL}. More precisely, there exists some constant $C$ only depending on $p$ such that $\mathbb{P}(M_{2n}=(0,0) ) \sim C n^{-\frac 5 4}$.
Since the random walk $M$ {\it does not} have the Markov property under the annealed law, we are not able to deduce the survival probability from the previous local limit theorem.
Nevertheless we will use the fact that $M$ has stationary increments under the annealed law).
Let us define the survival probability as the probability that the particle does not visit the $y-$axis (i.e. the line $x=0$) before time $n$, i.e.\ $\mathbb{P} (T_0^{(1)} >n)$, where 
$$T_0^{(1)}:=\inf\{n\ge 1\ :\ M^{(1)}_n=0\}$$
is the first return time of the first coordinate $M^{(1)}$ of $M$ to $0$.
Due to \cite{GPN}, the first coordinate $M^{(1)}_{\lfloor nt\rfloor}$ normalized by $n^{\frac 34}$ converges in distribution to $K_p\Delta_t^{(0)}$, where $K_p:=\frac p{(1-p)^{\frac 14}}$ and where $\Delta^{(0)}$ is the Kesten-Spitzer process $\Delta$
introduced in the previous sections
with $Y$ and $W$ two independent standard Brownian motions.
As for RWRS, the asymptotic behavior of this probability will be deduced from the range $\mathcal R_n^{(1)}$ of the first coordinate,  i.e.\ the number of vertical lines visited by
$(M_k)_k$ up to time $n$, namely
$$\mathcal R_n^{(1)}:=\#\{x\in\mathbb Z\ :\ \exists k=1,\ldots,n,\ \exists y\in\mathbb Z\ :\ M_k=(x,y)\}.$$

Our main result for the MdM random walk is the following.

\begin{thm}\label{thm:T01}
\begin{equation}\label{EM1}
\lim_{n\rightarrow +\infty} n^{-\frac 34}\mathbb E\left[\max_{k=1,\ldots,n}M_k^{(1)}\right]=K_p\mathbb E\left[\sup_{t\in[0,1]}\Delta_t^{(0)} \right].
\end{equation}
The sequence $(\mathcal R_n^{(1)}/n^{\frac 34})_{n\ge 1}$ converges in distribution
to $K_p\left(\sup_{t\in[0,1]}\Delta_t^{(0)}-\inf_{t\in[0,1]}\Delta_t^{(0)}
\right).$
Moreover
\begin{equation}\label{EEE2c}
\lim_{n\rightarrow+\infty} \frac{\EE[\mathcal R_n^{(1)}]}{n^{\frac 34}}=2K_p\, \mathbb E\left[\sup_{t\in[0,1]}\Delta^{(0)}_t\right]
\end{equation}
and
\begin{equation}\label{EEE3c}
\quad\lim_{n\rightarrow +\infty}  n^{\frac 1 4}\mathbb P(T_0^{(1)}>n)=\frac 32K_p\, \mathbb E\left[\sup_{t\in[0,1]}\Delta^{(0)}_t\right].
\end{equation}
\end{thm}

\begin{proof}
Recall that
$\mathcal R_n^{(1)}=\max\{M_1^{(1)},...,M_n^{(1)}\}+\min\{M_1^{(1)},...,M_n^{(1)}\}+1$.
It has been proved in \cite{GPN} that $((M^{(1)}_{\lfloor nt\rfloor}/n^{\frac 34})_t)_{n\ge 1}$ converges in distribution to $(K_p\Delta_t^{(0)})_t$ in the Skorohod space endowed with the $J_1$-metric. 
Hence $(n^{-\frac 34}(\max_{k=1,\ldots,n}M_k^{(1)}-\min_{\ell=1,\ldots,n}M_\ell^{(1)}))
_{n\ge 1}$ converges in distribution
to $K_p(\sup_{t\in[0,1]}\Delta_t^{(0)}-\inf_{s\in[0,1]}\Delta_s^{(0)})$.
Hence, to prove \eqref{EM1} and \eqref{EEE2c}, it is enough to prove that this sequence is
uniformly integrable. To this end we will prove that
it is bounded in $L^2$.
\\
Recall that the second coordinate of the MdM random walk is a lazy random walk. Let us denote it by $(S_n)_ {n\ge 0}$.
Observe that 
$$M_n^{(1)}:=\sum_{k=1}^n \xi_{S_k}\ind_{\{S_k=S_{k-1}\}}=\sum_{y\in\mathbb Z} \xi_{y}\tilde N_n(y),$$
with $\tilde N_n(y):=\#\{k=1,\ldots,n\ :\ S_k=S_{k-1}=y\}$.
Observe that $\tilde N$ is measurable with respect to the random walk $S$ and that $0\le \tilde N_n(y)\le N_n(y)$,
where again $N_n(y):=\#\{k=1,\ldots,n\ :\ S_k=y\}$.
\\
Conditionally to
the walk $S$, the increments of $(M^{(1)}_n)_ {n\ge 0}$ are centered and positively associated. It follows from Theorem 2.1 of \cite{Gong} that 
\[
\forall\beta\in[1,2),\quad
\mathbb E\left[\left|\max_{j=0,\ldots,n}M^{(1)}_j\right|^\beta\Big|S\right] 
 \le c_{2} \mathbb E\left[|M^{(1)}_n|^2|S\right]^{\beta/2}
 = c_{2} \left(\sum_{y\in\mathbb Z}(\tilde N_n(y))^2\right)^{\beta/2} \le c_{2} (V_n)^{\beta/2},
\]
where again $V_n=\sum_{y\in\mathbb Z}(N_n(y))^2$.
Therefore
$
\mathbb E\left[\left|\max_{j=0,\ldots,n}M^{(1)}_j\right|^\beta\right] \le  c_{2} (\mathbb E[V_n])^{\beta/2}\sim c'n^{\frac {3\beta}4}$
(see \cite[(1.2)]{KS}). This gives \eqref{EEE2c}.
Finally, \eqref{EEE3c} directly follows from \eqref{EEE2c} combined with \eqref{gammaa2} of Theorem \ref{THMZ}.
\end{proof}
\begin{rqe} In the historical model \cite{MdM}, the probability $p$ is equal to $1/3$, and in this particular case the survival probability is similar to $ \kappa n^{-\frac{1}{4}}$ 
where 
$\kappa =\left(\frac{3}{2^5}\right)^{1/4} \mathbb E\left[\sup_{t\in[0,1]}\Delta^{(0)}_t\right]$.
An open question is to compute the value of the above expectation. Numerical simulations give 
$\mathbb E\left[\sup_{t\in[0,1]}\Delta^{(0)}_t\right] \approx 0.54.$
\end{rqe}
\begin{thm}[Persistence probability for the MdM model]\label{mdm}
\[
\mathbb P\left(
         \max_{k=1,\ldots,n} M_k^{(1)}\le -1\right)\sim \frac 34\frac{p}{(1-p)^{\frac 14}}\mathbb E\left[\sup_{t\in[0,1]}\Delta_t^{(0)}\right] n^{-\frac 14},\quad
        \mbox{as }n\rightarrow +\infty.
\]
\end{thm}
This result was conjectured by Redner \cite{Red2} and Majumdar \cite{Maj}.
\begin{proof}
This comes directly from
Theorem \ref{thm:T01} and from the last item of Theorem \ref{THMZ}.
\end{proof}
\begin{rqe}[Range in the historical MdM model]\label{cvpsrange}
It is worth noticing that the range ${\mathcal R}_n$ of the MdM random walk, i.e.\ the number of sites visited by $M$ before time $n$, 
${\mathcal R}_n:=\#\{M_0,\dots, M_{n}\}$ is well understood.
Using 
\cite[Lemma 3.3.27]{Zeitouni}, $({\mathcal R}_n/n)_{n\ge 1}$ converges $\mathbb P$-almost surely to $\mathbb P[M_j\ne 0,\ \forall j\geq 1]$, which contradicts the result announced in \cite{LeNy}.
(There, we consider the ergodic dynamical system $(\Omega,\mu,T)$ given by $\Omega:=\{-1,0,1\}^{\mathbb Z}\times
\{\pm1\}^{\mathbb Z}$, $\mu:=(\mathbb P_{S_1})^{\otimes\mathbb Z}\otimes (\mathbb P_{\xi_1})^{\otimes\mathbb Z}$ and 
$T((\alpha_k)_k,(\epsilon_k)_k):=((\alpha_{k+1})_k,
  (\epsilon_{k+\alpha_0})_k)$ (see for instance \cite{KMC} for its ergodicity, p.\ 162). We set
$f((\alpha_k)_k,(\epsilon_k)_k)=(\epsilon_0 {\bf 1}_{\{\alpha_0=0\}}, \alpha_0)$.
With these choices, $(M_j)_{j\geq 1}$ has the same distribution under $\mathbb P$
as $(\sum_{k=1}^j f\circ T^k)_{j\geq 1}$ under $\mu$.)
\end{rqe}

{\bf Acknowledgement.} We would like to thank the two referees and the associate editor for carefully reviewing our manuscript. Their suggestions significantly improved the presentation of the paper.\\
We are grateful to Fr\'ed\'erique Watbled for her questions
that made us realize some mistakes in our published paper.


\begin{thebibliography}{00}
\bibitem{Aurzada} Aurzada, F. {\it On the one-sided exit problem for fractional Brownian motion.}
Electron. Commun. Probab., 16:392--404, 2011.


\bibitem{aurzadadereich} Aurzada, F.\ and Dereich, S. {\it Universality of the asymptotics of the one-sided exit problem for integrated processes.} Ann.\ Inst.\ H.\ Poincar\'e Probab.\ Statist.\ 49(1):236--251, 2013.


\bibitem{AS} Aurzada, F.; Simon, T.  {\it Persistence probabilities \& exponents.} L\'evy matters V, p.\ 183-221, Lecture Notes in Math., 2149, Springer, 2015.

\bibitem{Bo89} Bolthausen, E. {\it A central limit theorem for two-dimensional random walks in random sceneries.}
Ann.\ Probab.\ 17 (1989) 108--115.

\bibitem{BGKPS} Bouchaud, J.P.; Georges, A.; Koplik, J.; Provata, A.; and Redner, S. {\it Superdiffusion in random velocity fields.} Phys. Rev. Lett. {\bf 64} (1990), 2503 - 2506.

\bibitem{Borodin} Borodin, A. N. {\it A limit theorem for sums of independent random variables defined 
on a recurrent random walk.}
(Russian)  Dokl. Akad. Nauk SSSR 246(4):786--787, 1979. 

\bibitem{BMS13}
Bray, A.~J.; Majumdar, S.~N.; and Schehr, G.
\newblock {\it Persistence and first-passage properties in non-equilibrium systems.}
\newblock  Advances in Physics, 62(3):225--361, 2013.

\bibitem{CP} Campanino, M. and P\'etritis, D. {\it Random walks on randomly oriented lattices.}  Mark. Proc. Relat. Fields (2003), 9, 391--412.

\bibitem{FFN} Castell, F.; Guillotin-Plantard, N. and P\`ene, F. {\it Limit theorems for one and two-dimensional random walks in random scenery.} Annales de l'Institut Henri Poincar\'e - Probabilit\'es et Statistiques (2013), Vol. 49, No 2, 506--528.

\bibitem{TLL} Castell, F.; Guillotin-Plantard, N.; P\`ene, F.; and Schapira, B. 
{\it A local limit theorem for random walks in random scenery and on randomly oriented lattices}.  Annals of Probability 39 (6), 2079--2118, 2011.

\bibitem{BFFN} Castell, F.; Guillotin-Plantard, N.; P\`ene, F.; and Schapira, B. 
{\it On the one-sided exit problem for stable processes in random scenery.} 
Electron. Commun. Probab. 18(33):1--7, 2013.

\bibitem{castellguillotinwatbled} Castell, F.; Guillotin-Plantard, N.; and Watbled, F. {\it Persistence exponent for random processes in Brownian scenery}. 
ALEA, Lat. Am. J. Probab. Math. Stat. (2016), Vol. 13, 79--94.


\bibitem{NadineClement} Dombry, C. and Guillotin-Plantard, N. {\it Discrete approximation of a  stable  self-similar stationary  increments  process.} 
Bernoulli (2009), Vol. 15, No 1, 195--222.


\bibitem{Gong} Gong, X. {\it Maximal $\phi$-inequalities for demimartingales.}, J. Inequal. Appl. 2011, 2011:59, 10 pp.



\bibitem{GPN} Guillotin-Plantard, N. and Le Ny, A. {\it  A functional limit theorem for a 2d- random walk with dependent marginals.}
Electronic Communications in Probability (2008), Vol. {\bf 13}, 337--351. 

\bibitem{GuPo} Guillotin-Plantard, N. and Poisat, J. {\it  Quenched central limit theorems for random walks in random scenery}.
Stochastic Process. Appl. 123 (4) (2013) 1348--1367.

\bibitem{KMC} Kalikow, S. and McCutcheon, R. {\it An outline of ergodic theory.} Cambridge University Press (2010).

\bibitem{Karamata} Karamata, J. {\it Sur un mode de croissance r\'eguli\`ere. Th\'eor\`emes fondamentaux},
\newblock Bulletin de la Soci\'et\'e Math\'ematique de France, 61 (1933), p. 55--62 



\bibitem{KS} Kesten, H. and Spitzer, F. 
{\it  A limit theorem related to a new class of self-similar processes.} 
Z. Wahrsch. Verw. Gebiete 50:5--25, 1979.  
%

\bibitem{KL} Khoshnevisan, D. and Lewis, T. M. {\it A law of iterated logarithm for stable processes in random scenery.}
Stochastic Process. Appl., 74(1):89--121, 1998.

\bibitem{LGR} Le Gall, J.F. and Rosen, J.  {\it The range of stable random walks.} Ann. Probab. {\bf 19} (1991), 650--705.

\bibitem{LeNy} Le~Ny, A. {\it Range of a Transient 2d-Random Walk} (2011) ArXiv:1111.0877.



\bibitem{Lifshits} Lifshits, M.A. Gaussian random functions. Kluwer Academic Publishers, 1995.

\bibitem{Maj1} Majumdar, S. {\it Persistence in nonequilibrium systems}. Current Science 77 (3):370-375, 1999.

\bibitem{Maj}  Majumdar, S. {\it Persistence of a particle in the Matheron - de Marsily velocity field.} 
Phys. Rev. E 68, 050101(R), 2003.

\bibitem{marcusrosen} Marcus, M. B. and Rosen, J. Markov processes, Gaussian processes, and local times. Cambridge Studies in Advanced Mathematics, 100. Cambridge University Press, Cambridge, 2006.

 \bibitem{MdM} Matheron, G. and de Marsily G.
 {\it Is transport in porous media always diffusive? A counterexample.}
 Water Resources Res. 16:901--907, 1980. 



\bibitem{Molchan1999} Molchan, G.M. {\it Maximum of fractional Brownian motion: probabilities of small values.}
Comm. Math. Phys., 205(1):97--111, 1999.
%


%
%








\bibitem{Red2}  Redner, S. {\it Survival Probability in a Random Velocity Field.}  Phys. Rev.,  E56,  4967 (1997).


 \bibitem{samorodnitsky} Samorodnitsky, G. {\it Long range dependence.}
Found. Trends Stoch. Syst. 1 (2006), no. 3, 163--257.

\bibitem{Spitzer}Spitzer, F. {\it Principles of random walks}
 Princeton (N.J.) and London,. Van Nostrand, 1964. xi, 406 p.






\bibitem{taqqu} Taqqu, M.S. {\it Weak convergence to fractional Brownian motion and to the Rosenblatt process.} Z. Wahrscheinlichkeitstheorie und Verw. Gebiete 31 (1974/75), 287--302.

\bibitem{VBE} Von Bahr, B. and Esseen, C-G, {\it Inequalities for the rth Absolute Moment of a Sum of Random Variables, $1\le r\le 2$.} Ann. Math. Statist.
Vol. 36, No 1 (1965), 299--303.

\bibitem{Wang}  Wang, W. {\it Weak convergence to fractional Brownian motion in Brownian scenery}.
 Probab. Theory  and Related  Fields, 126(2):203-220, 2003.
 
 \bibitem{Whitt} Whitt, W. {\rm Stochastic process limits.}  Springer series in Operations research, Springer Verlag, New York, 2002.

\bibitem{Zeitouni} Zeitouni, O. {\rm Random walks in random environment.} Lectures on probability theory and statistics. Springer Berlin Heidelberg, (2004), 189--312.

%
%
%

\end{thebibliography}
\end{document}